\documentclass{article}
\usepackage{verbatim} 

\usepackage{graphicx, enumitem, amsmath, amsfonts, amssymb, tcolorbox, amsthm,hyperref,fullpage}

\newcommand{\del}{\partial}
\newcommand{\C}{\mathbb{C}}
\newcommand{\Z}{\mathbb{Z}}

\newcommand{\SL}{\mathrm{SL}}

\renewcommand{\phi}{\varphi}

\renewcommand{\gcd}{\text{gcd}}

\usepackage[parfill]{parskip}

\newcommand{\e}{\varepsilon}

\renewcommand{\and}{\text{ and }}

\newcommand{\id}{\text{id}}
\newcommand{\fourmat}[4]{\begin{bmatrix}#1&#2\\#3&#4\end{bmatrix}}
\newcommand{\lilmat}[4]{\left[\begin{smallmatrix}#1&#2\\#3&#4\end{smallmatrix}\right]}

\newcommand{\into}{\hookrightarrow}
\newcommand{\onto}{\twoheadrightarrow}
\newcommand{\lildet}[4]{\left|\begin{smallmatrix}#1&#2\\#3&#4\end{smallmatrix}\right|}

\renewcommand{\v}{\mathbf{v}}

\renewcommand{\a}{\alpha}

\renewcommand{\b}{\beta}
\newcommand{\g}{\gamma}

\renewcommand{\S}{\mathcal{S}}
\newcommand{\SRY}{\mathcal{S}^{RY}}
\newcommand{\san}{\Sigma_{0,2,2}}
\newcommand{\sto}{\Sigma_{1,0,0}}
\newcommand{\twos}[2]{\left[\begin{smallmatrix}#1\\#2\end{smallmatrix}\right]}
\newcommand{\two}[2]{\begin{bmatrix}#1\\#2\end{bmatrix}}

\newtheorem{theorem}{Theorem}[section]
\newtheorem{lemma}[theorem]{Lemma}
\newtheorem{proposition}[theorem]{Proposition}
\newtheorem{corollary}[theorem]{Corollary}

\theoremstyle{definition}
\newtheorem{example}[theorem]{Example}
\newtheorem{remark}[theorem]{Remark}
\newtheorem{definition}[theorem]{Definition}
\newtheorem{problem}[theorem]{Problem}

\renewcommand{\arraystretch}{1.5}

\renewcommand{\g}{\gamma}
\renewcommand{\a}{\alpha}
\renewcommand{\b}{\beta}
\renewcommand{\san}{\Sigma_{0, 2, 2}}
\newcommand{\torus}{\Sigma_{1, 0, 0}}
\renewcommand{\SRY}{\ensuremath{ \mathcal S^{RY}}}

\title{A relationship between the Kauffman bracket skein algebras and Roger-Yang skein algebras of some small surfaces}
\author{Chloe Marple, Helen Wong}
\date{}

\begin{document}
\maketitle

\def\arraystretch{1.2}

\begin{abstract}
We calculate the Roger-Yang skein algebra of the annulus with two interior punctures, $\SRY(\san)$, and show there is a surjective homomorphism from this algebra to the Kauffman bracket skein algebra of the closed torus. Using this homomorphism, we characterize the irreducible, finite-dimensional representations of $\SRY(\san)$, showing that they can be described by certain complex data and that the correspondence is unique if certain polynomial conditions are satisfied. We also use the relationship with the skein algebra of the torus to compute structural constants for a bracelets basis for $\SRY(\san)$, giving evidence for positivity.
\end{abstract} 

\section{Introduction}

This paper investigates the algebraic structure and representation theory of the Roger-Yang skein algebra of a punctured surface, in the specific case of a twice-punctured annulus,  and relates this algebra to the skein algebra of another small surface –- the closed torus.

The Roger-Yang skein algebra is a construction that bridges quantum topology with hyperbolic geometry, extending the definition of the Kauffman bracket skein algebra.  The Kauffman bracket skein algebra was originally defined as a generalization of the Jones polynomial to 3-manifolds \cite{PrzSkeinAlg, TuraevSkeinAlg}, but later found connections with many other areas of mathematics.  It not only plays a key role in topological quantum field theories \cite{BHMV, BCGP}, but it is also a deformation quantization of the $\mathrm{SL}_2(\C)$-character variety of the surface,  which contains a copy of the Teichmüller space \cite{TuraevPoisson, BullockRings, BullockFrohmanJKB, PrzSikora}.  
More recently, a deeper understanding of  its multiplicative structure has led researchers to use algebraic geometric techniques  \cite{FKBLUnicity, GJSUnicity, KKUnicity}
to study its representation theory \cite{BonWonQTrace, BonWonSkeinReps1}.  

In \cite{RogerYang}, Roger and Yang  generalized the Kauffman bracket skein algebra for a surface with interior punctures so that generators included both framed loops and arcs between punctures in the thickened surface.  Skein relations were designed to capture the combinatorics of $\lambda$-lengths for Penner's decorated Teichmüller space of a punctured surface \cite{PennerDecoratedTeichCMP, PennerDecoratedTeichJDG}, and indeed, Roger and Yang proposed that their skein algebra should be a deformation quantization of Penner's decorated Teichmüller space of a punctured surface.  This was later verified  by relating the Roger-Yang skein algebra  with Fomin-Shapiro-Thurston's cluster algebra for punctured surfaces \cite{MoonWongQT, MoonWongComp}, or alternatively by relating it with a suitably defined quantum torus \cite{BKLQTrace}.   

It is conjectured that the representation theory of the Roger-Yang skein algebra should encode hyperbolic geometric information about the punctured surface.  In the analogous situation of the Kauffman bracket skein algebra for a surface, there is a Unicity Theorem that says that a Zariski open dense subset of the $\mathrm{SL}_2(\C)$-character variety of the surface where points in the subset are in one-to-one correspondence with an irreducible, finite-dimensional representation of the Kauffman bracket skein algebra
 \cite{FKBLUnicity, GJSUnicity, KKUnicity}.  Moreover, there are general constructions of irreducible, finite dimensional representations from geometric data \cite{BonWonSkeinReps2, BonWonSkeinReps3}, and for some small surfaces, these  representations have been explicitly computed \cite{TakenovReps, YuReps}.  In contrast, much less is known about the representation theory of the Roger-Yang skein algebra when the surface has at least two interior punctures (and no marked points on the boundary);  see \cite{KaruoMoonWongCenter} for further discussion and references.  

The goal of this paper is to examine one of the simpler non-trivial Roger-Yang skein algebra,  namely that of the twice-punctured annulus. Let us denote it by $\mathcal S^{RY}(\san)$. We first determine a finite presentation for $\mathcal S^{RY}(\san)$ where the generators can be chosen to be ordered (as in \cite{FrohmanAbdielAGT}); see Section \ref{sec:pres}.    Note that presentations are known only for a few cases for the usual Kauffman bracket skein algebra \cite{BullockPrz, CookeGenus2, CookeAskeyWilson} and the Roger-Yang skein algebra \cite{BobbPeiferInvolve, MoonPuncSph}.

Our calculation of the presentation of $\mathcal S^{RY}(\san)$ revealed a close relationship with the skein algebra of the one-holed torus, denoted $\mathcal S(\torus)$. In particular, there is a surjective homomorphism  of $\mathcal S^{RY}(\san)\to \mathcal S(\sto)$, and the kernel is nontrivial. In Section \ref{sec:relationships}, we construct this homomorphism, and give another example of such a relationship, between the Roger-Yang skein algebra of the thrice-punctured annulus and the Kauffman bracket skein algebra of the closed torus. In a forthcoming paper, we will investigate whether there is a general phenomenon explaining these relationships.

The existence of a surjection $\mathcal S^{RY}(\san)\to \mathcal S(\sto)$ led us to ask to what extent can we pull back results for the closed torus to the twice-punctured annulus. We were able to do so for two key pieces of algebraic information about $\mathcal S^{RY}(\san)$: its representations and its structural constants.

In Section \ref{sec:reps}, we study the representation theory of $\mathcal S^{RY}(\san)$ following the work of \cite{FKBLUnicity, TakenovReps, HavlicekPosta}  for the closed torus and $U_q(so_3)$.  We first show that, like the usual Kauffman bracket skein algebra,  $\mathcal S^{RY}(\san)$ is almost Azumaya, which implies that almost every irreducible, finite-dimensional representation of $\mathcal S^{RY}(\san)$ is determined by its central character.   When $A$ is a primitive $2N$ root of unity with $N$ odd,  the central character can be described by a complex 5-tuple, which we refer to as its classical shadow data.  We provide  polynomial conditions under which the classical shadow data uniquely determines a representation.  We also provide explicit computations of the representations, which are all $N$-dimensional.   

In Section \ref{sec:pos}, we give evidence of the positivity of structural constants for the bracelet basis for $\mathcal S^{RY}(\san)$, as conjectured for the usual Kauffman bracket skein algebra by Thurston in \cite{ThurstonPositive} and then extended to a specialized Roger-Yang skein algebra for surfaces with interior punctures by Karuo in \cite{KaruoPositive}. The version of the  positivity conjecture for skein algebras for surfaces without interior punctures was proven by \cite{MandelQinPositive} by applying a positivity result for cluster algebras \cite{GHKPositive}.  Other approaches have appealed to curve-counting methods related to Gromov-Witten theory \cite{BousseauPositive}.  Here, our method for the case with interior punctures is more elementary and based on skein theory, as in \cite{LePositive, TTYPositive, Rhea4puncsphere}.   By Frohman-Gelca's Product-to-Sum formula for the closed torus \cite{FrohmanGelca} and  the method of \cite{SikiFast}, we deduce the structural constants for infinitely many basis curves of $\mathcal S^{RY}(\san)$. While we were unable to prove it in general, our  results provide strong computational evidence for Karuo's positivity conjecture for $\mathcal S^{RY}(\san)$.   
  
\section{Acknowledgements}
We gratefully acknowledge the support of the mathematics departments of  Claremont McKenna College and Pomona College during this research.  We also thank Francis Bonahon, Hiroaki Karuo, and Han-Bom Moon for valuable feedback and suggestions on this research.    The authors were funded in part by DMS-2305414 from the US National Science Foundation.

\section{Definition of Roger-Yang skein algebra}
Let $\Sigma=\Sigma_{g,b, p}$ be the compact, oriented genus $g$ surface with $b$ boundary components with $p$ punctures.  Let $\del_0, \ldots, \del_b$  denote the boundary components and $v_1, \ldots v_p $ the punctures (sometimes also referred to as interior marked points). 
 
A \emph{framed link in $\Sigma \times [0,1]$} is a disjoint union of finitely many framed knots and framed arcs ending at the punctures $v_1, \ldots v_p $, regarded up to regular isotopy. For details, see \cite{RogerYang}.   We regard  the second coordinate of $\Sigma \times [0,1]$ as describing height, and take the convention that if an arc ends at $v_i$, the framing of the arc at $v_i$ is pointing towards increasing height.    Note that many framed arcs may end at a puncture $v_i$, but they must do so at differing heights along $v_i \times[0,1]$.  

We usually describe framed links using a diagram, that is, a projection of the framed link which is isotoped into general position. This means there are only transversal double points except at the punctures and the framing is always vertical.  Over- and under-crossings at the double points are indicated by  breaks in the projection.  If there are more than two ends of arcs at a puncture, they are further labelled to show the ordering by height.  Regular isotopy of the framed links in $\Sigma$ can be described using certain moves on their diagrams, as depicted in Figure~\ref{fig:Rmoves}. 

\begin{figure} 
\[ 		
\begin{minipage}{.07\textwidth}
		\includegraphics[width=\textwidth]{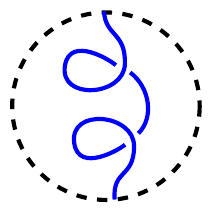}
		\end{minipage}
 \leftrightarrow
		\begin{minipage}{.07\textwidth}
		\includegraphics[width=\textwidth]{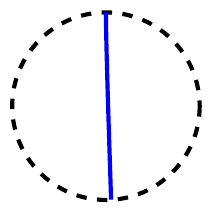}
		\end{minipage},
\quad
 \begin{minipage}{.07\textwidth}
		\includegraphics[width=\textwidth]{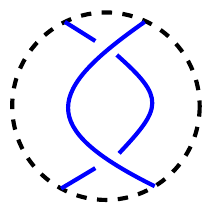}
		\end{minipage}
 \leftrightarrow
		\begin{minipage}{.07\textwidth}
		\includegraphics[width=\textwidth]{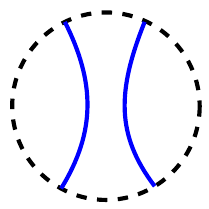}
		\end{minipage},
\quad
 \begin{minipage}{.07\textwidth}
		\includegraphics[width=\textwidth]{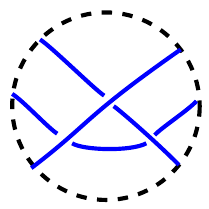}
		\end{minipage}
 \leftrightarrow
		\begin{minipage}{.07\textwidth}
		\includegraphics[width=\textwidth]{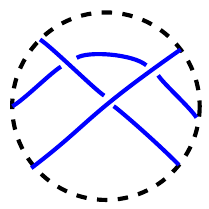}
		\end{minipage},		
\quad
 \begin{minipage}{.07\textwidth}
		\includegraphics[width=\textwidth]{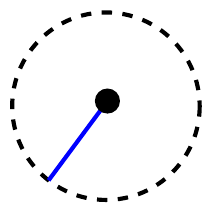}
		\end{minipage}
 \leftrightarrow
		\begin{minipage}{.07\textwidth}
		\includegraphics[width=\textwidth]{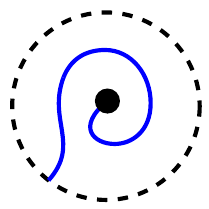}
		\end{minipage}, 
\quad 
 \begin{minipage}{.07\textwidth}
		\includegraphics[width=\textwidth]{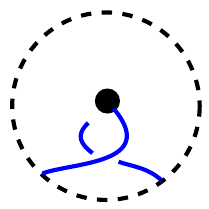}
		\end{minipage}
 \leftrightarrow
		\begin{minipage}{.07\textwidth}
		\includegraphics[width=\textwidth]{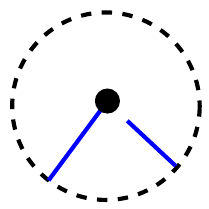}
		\end{minipage}
\]
\caption{Framed links that are equivalent up to regular isotopy} \label{fig:Rmoves}
\end{figure}

Given two framed links  $\alpha, \beta$ in $\Sigma \times [0,1]$, we may stack $\alpha$ on top of $\beta$ to obtain $\alpha * \beta$.  In particular, $\alpha * \beta$ is the union of the framed curve $\alpha' \subset  \Sigma \times [0, \frac12]$ (obtained by rescaling $\alpha$ in $  \Sigma \times [0, 1]$ vertically by half) and of  the framed curve $\beta' \subset  \Sigma \times [\frac12, 1]$ (obtained by rescaling $\beta$ in $\Sigma \times [0, 1]$ vertically by half).     

Let $R$ be a commutative domain, and $A \in R$ be an invertible element with distinguished square roots $A^{\pm 1/2}$.  For every $i = 1, \ldots p$, we identify the $i$th puncture with the indeterminate variable $v_i$. 

\begin{definition}\label{def:RYalgebra}
The \emph{Roger-Yang skein algebra $\SRY(\Sigma)$} is the $R[v_1^{\pm 1} , \ldots v_p^{\pm 1} ]$-algebra freely generated by framed links in $\Sigma \times [0,1]$ modded out by the following relations:
\begin{align*}
&1)
\quad
\begin{minipage}{.5in}\includegraphics[width=\textwidth]{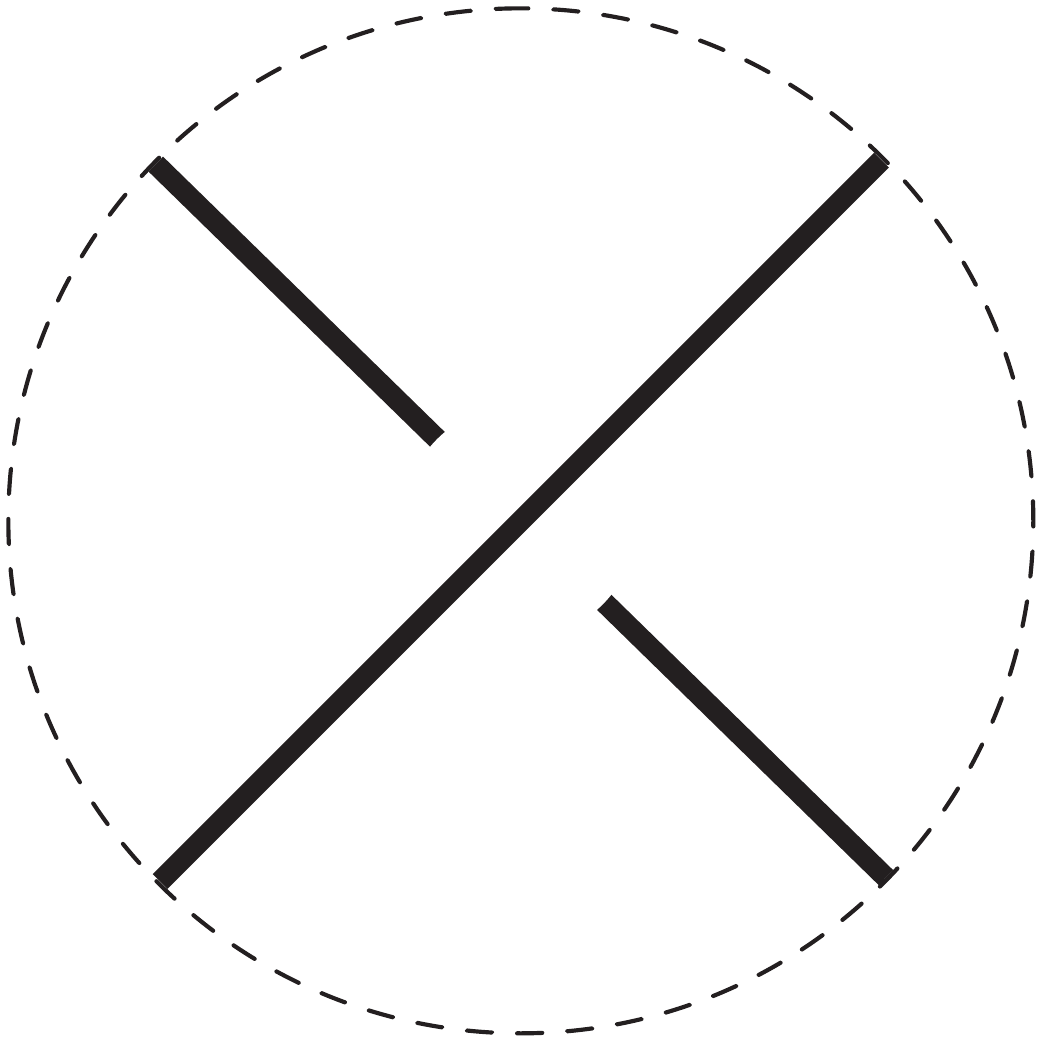}\end{minipage} 
-  \left( A\begin{minipage}{.5in}\includegraphics[width=\textwidth]{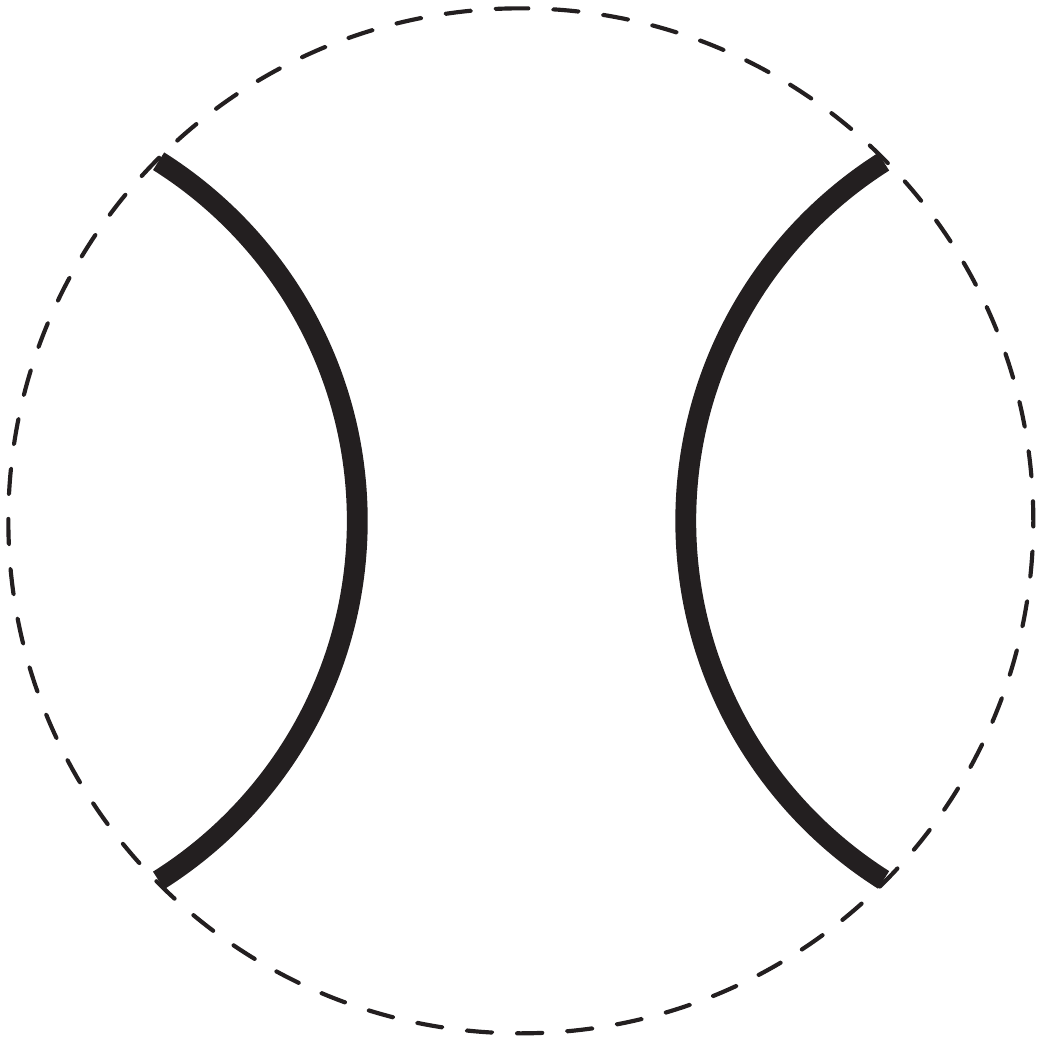}\end{minipage} 
+A^{-1}\begin{minipage}{.5in}\includegraphics[width=\textwidth]{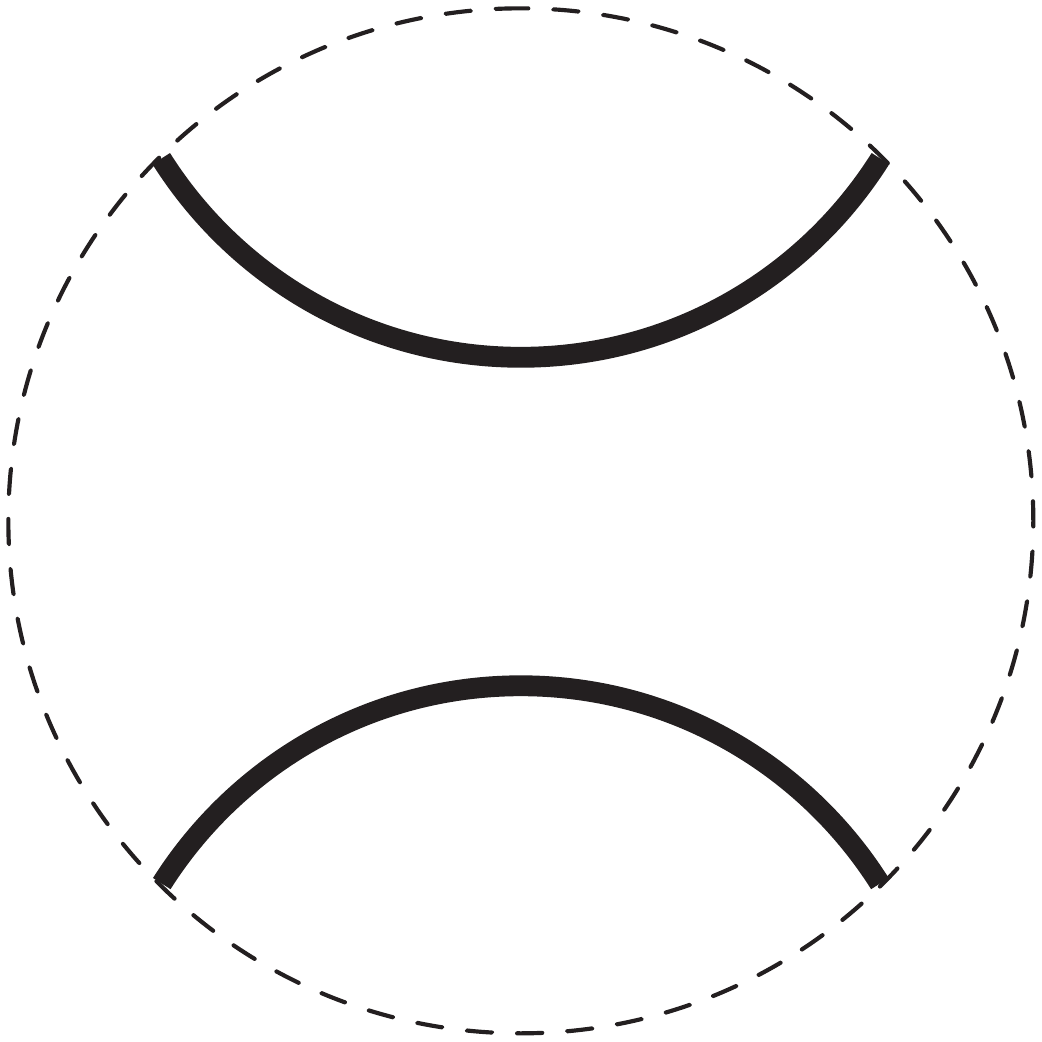}\end{minipage}  \right)\\
&2)
\quad 
v_i \begin{minipage}{.5in}\includegraphics[width=\textwidth]{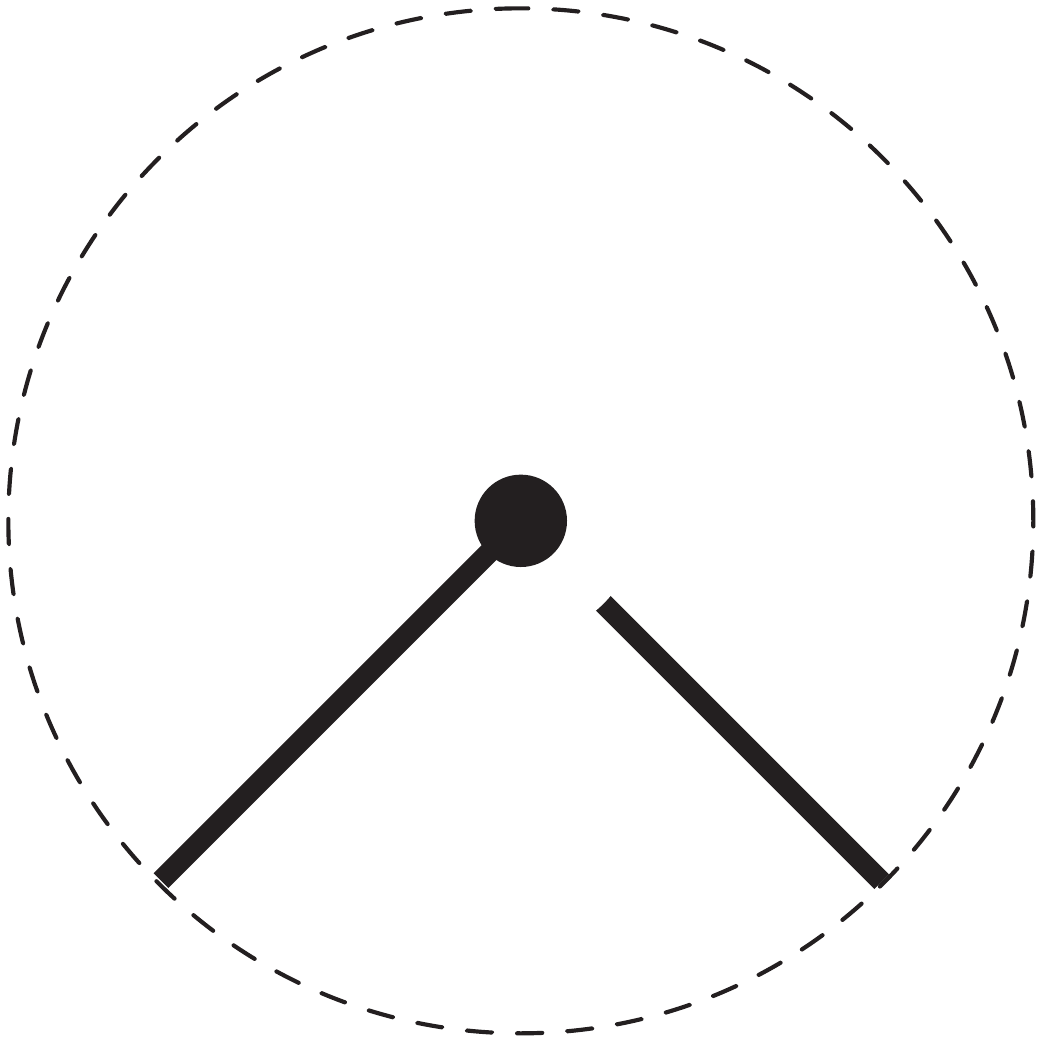}\end{minipage} 
- \left(  A^{1/2}\begin{minipage}{.5in}\includegraphics[width=\textwidth]{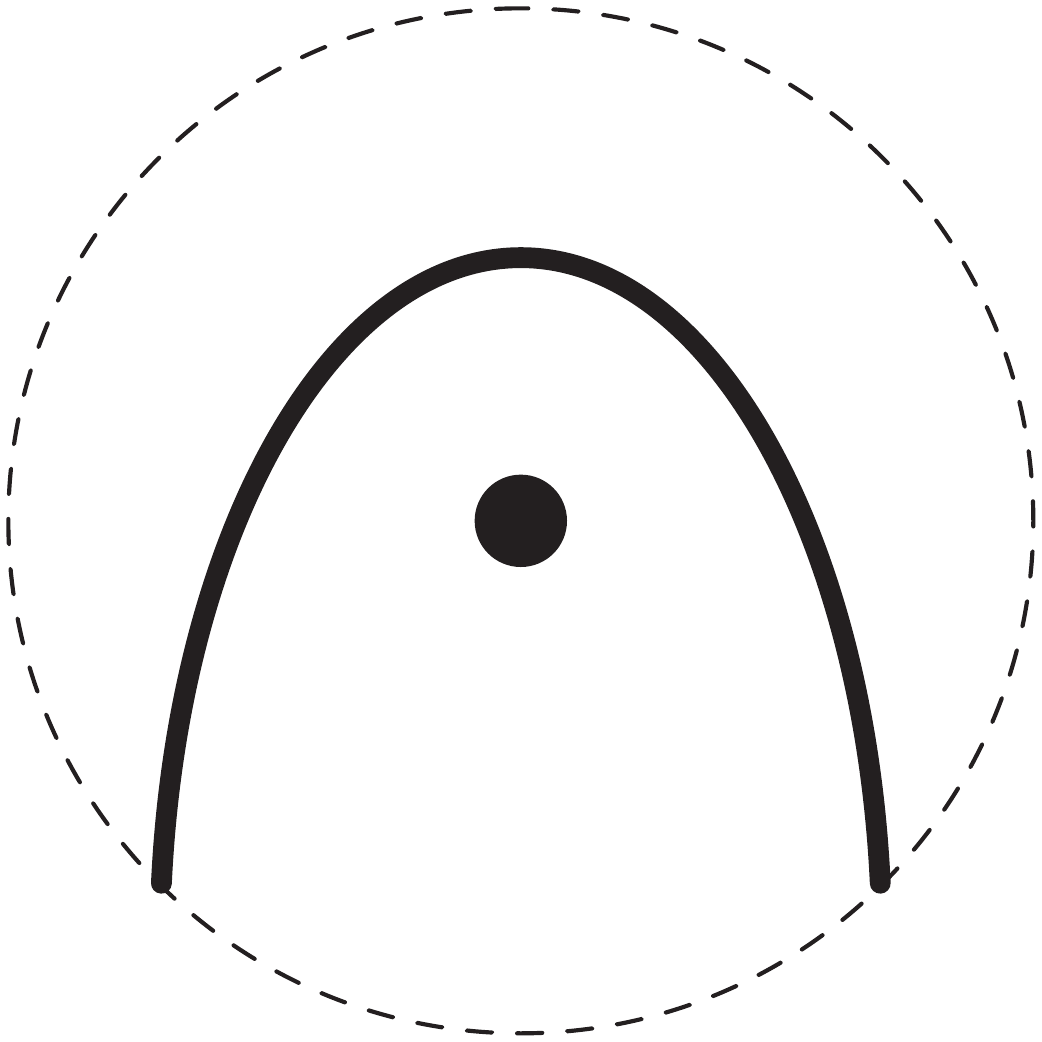}\end{minipage} 
+ A^{-1/2}\begin{minipage}{.5in}\includegraphics[width=\textwidth]{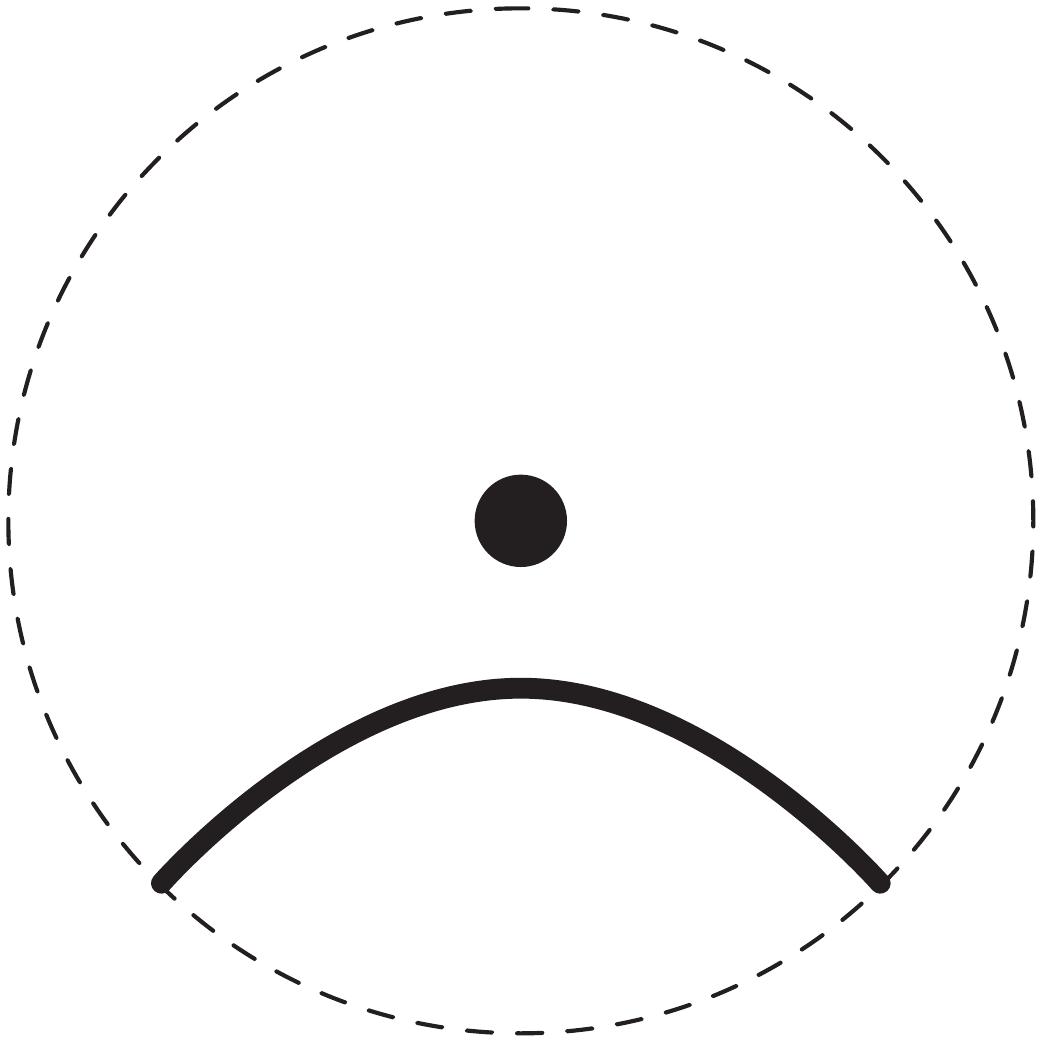}\end{minipage}  \right)\\
&3)
\quad 
\begin{minipage}{.5in}\includegraphics[width=\textwidth]{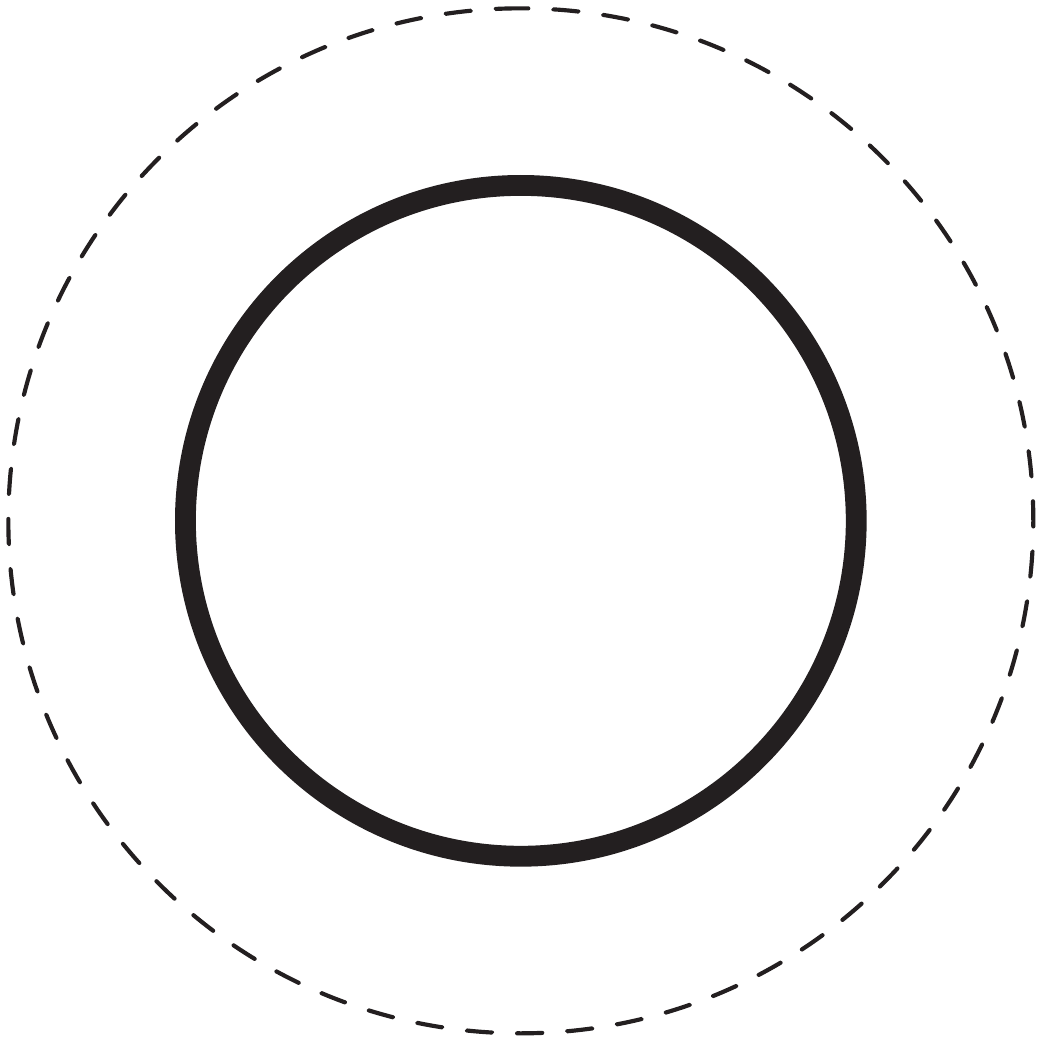} \end{minipage} 
- ( - A^2 - A^{-2} )\\
&4)
\quad 
\begin{minipage}{.5in}\includegraphics[width=\textwidth]{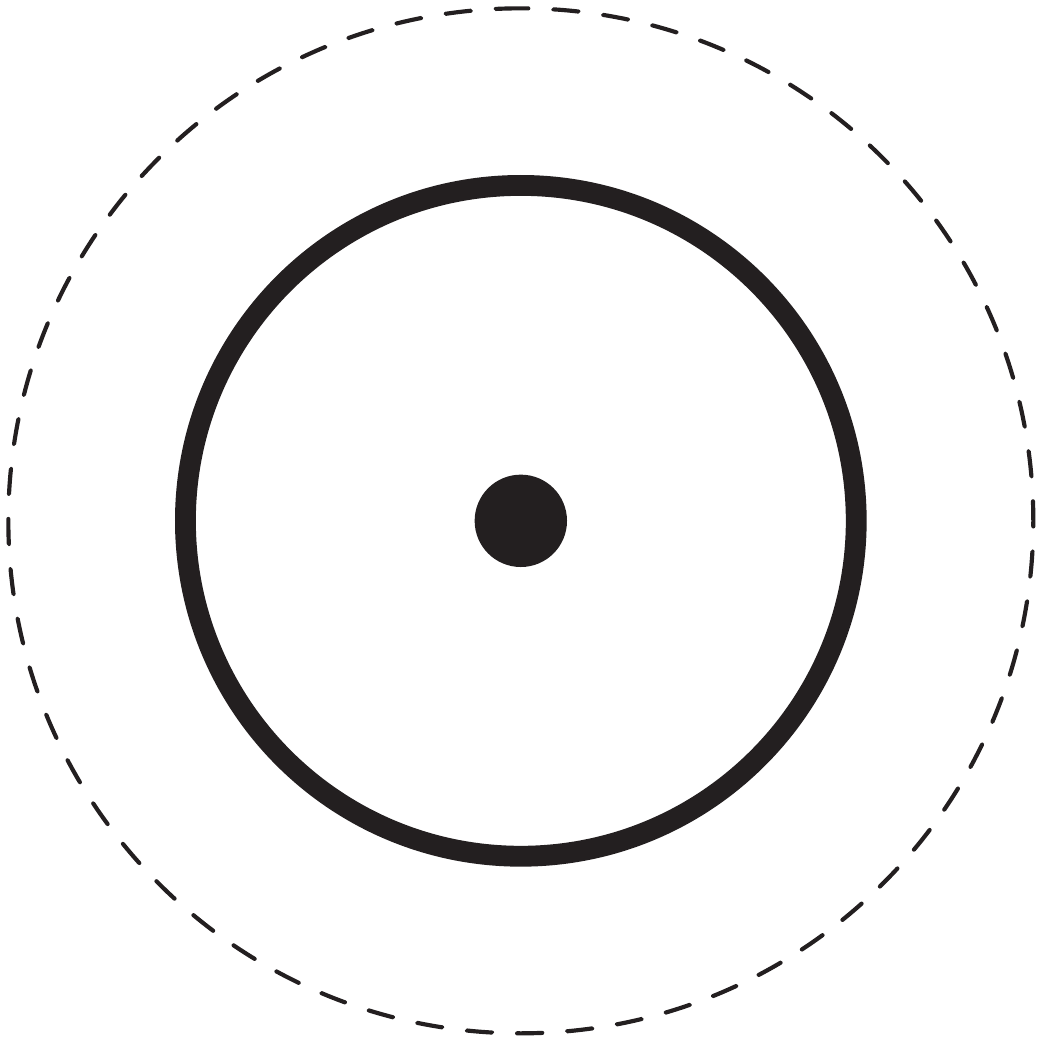} \end{minipage} 
-( A + A^{-1} ),\\ 
\end{align*}
 where the diagrams in the relations are assumed to be identical outside of the small balls depicted.  Multiplication of elements in $\S^{RY}(\Sigma)$ is the one induced by the stacking operation for framed links. 
\end{definition}

Henceforth, we will take $R=\Z[A^{\pm 1/2}]$, where $A$ is an indeterminate.

\begin{example}\label{ex:T2arc}
The square of an arc between punctures $i$ and $j$ is resolved as follows.
\begin{align*}
v_iv_j\begin{minipage}{.4in}\includegraphics[width=\textwidth]{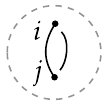} \end{minipage} 
=A\begin{minipage}{.4in}\includegraphics[width=\textwidth]{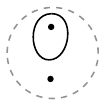} \end{minipage} 
+A^{-1}\begin{minipage}{.4in}\includegraphics[width=\textwidth]{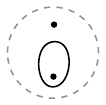} \end{minipage} 
+\begin{minipage}{.4in}\includegraphics[width=\textwidth]{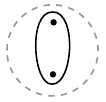} \end{minipage}
+\begin{minipage}{.4in}\includegraphics[width=\textwidth]{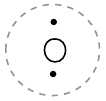} \end{minipage}  
=(A+A^{-1})^2-A^2-A^{-2}+\begin{minipage}{.4in}\includegraphics[width=\textwidth]{figures/both.pdf} \end{minipage} 
=2+\begin{minipage}{.4in}\includegraphics[width=\textwidth]{figures/both.pdf} \end{minipage} 
\end{align*}
\end{example}

Note that, in contrast, the usual Kauffman bracket skein algebra of $\S(\Sigma)$ does not include vertex classes or framed arcs.  More specifically, $\S(\Sigma)$ is the $R$-algebra freely generated by disjoint unions of framed knots in $\Sigma \times [0,1]$ modded out by relations 1) and 3) from Definition \ref{def:RYalgebra}.   

 We next state some results about the  structure of $\SRY(\Sigma)$ that will be used in the paper.

 We say that a framed link in $\Sigma \times [0,1]$ is  \emph{simple} if it admits a diagram that contains no crossings and does not contain any loop bounding a disk with one or no punctures.  The empty link, denoted $\emptyset$, is considered a simple link.
 
 \begin{proposition}[ \cite{RogerYang}]
 The simple framed links freely span $\SRY(\Sigma)$ as an $R$-module. 
 \end{proposition}

\begin{proposition}[ \cite{BKLQTrace} ]
$\SRY(\Sigma)$ is orderly finitely generated as an algebra, meaning that there exists skein elements $a_1,\dots ,a_k $ such that the set of all ordered monomials $\{ a_1^{n_1}\cdots  a_k^{n_k} \}_{n_i \in \Z_{\ge 0}}$ spans $\SRY(\Sigma)$.
\end{proposition} 

It is an open question whether there is an orderly generating set consisting of simple curves.  In particular, it is not known whether the orderly generating set from \cite{BKLQTrace} are simple curves.  On the other hand, there is a generating set from  \cite{BobbPeiferFinGen} consisting of simple curves, but  it is not known whether they form an orderly finitely generating set for most surfaces.  In Section \ref{sec:pres}, we will show that $\mathcal S^{RY}(\san)$ is orderly generated using simple curves.

\begin{proposition}[ \cite{MoonWongComp, BKLQTrace} ]
$\SRY(\Sigma)$ is an integral domain.
\end{proposition}

Let  $T_k$ be the Chebyshev polynomial of the first kind, defined recursively by $T_0(x) = 2$, $T_1(x) = x$, and $T_{k+1}(x) = x T_{k}(x) - T_{k-1}(x)$.

\begin{proposition}[\cite{KaruoMoonWongCenter}, following \cite{BonWonSkeinReps1, LeMoreRoots, KaruoPositive}] \label{prop:center}
If $A^{2N} = 1$ with $N$ odd, the center of the Roger-Yang skein algebra $Z(\mathcal S_{A}^{RY}(\Sigma))$ contains the following skeins: 
\begin{enumerate}
\item $T_{N}(\alpha)$, \; where $\alpha$ is a loop class  without self-intersection on its diagram
\item  $\frac{1}{\sqrt{v}\sqrt{w}}T_{N}(\sqrt{v}\sqrt{w}\beta)$, \; where $\beta$ is an arc class connecting two distinct interior punctures $v$ and $w$ and does not admit any self-intersection on its diagram 
\item  $\delta$ where $\delta$ is any curve parallel to a component of $\del \Sigma$
\end{enumerate}

Moreover, if $A$ is a primitive root of unity of odd order $N$, then the center $Z(\mathcal S_{A}^{RY}(\Sigma))$ is equal to the $\mathbb C[v_{i}^{\pm}]$-subalgebra generated by the above skeins. 
\end{proposition}

Finally, we introduce some notation.   
\begin{definition} \label{def:comm}
Given $\alpha, \beta \in \SRY(\Sigma)$, define $[\alpha, \beta] =  \alpha * \beta -  \beta * \alpha$ and $[\alpha, \beta]_A = A \alpha * \beta - A^{-1} \beta * \alpha$. 
\end{definition}

\bigskip

\section{Presentations of $\SRY(\Sigma_{0,2,2})$} \label{sec:pres}

The first presentation we provide involves one of the boundary loops as well as the four generators shown in Figure~\ref{fig:022-gens}.  
The relations from Theorem~\ref{thm:022pres} are chosen for ease of proof rather than brevity.   A more compact version will be presented later as Corollary~\ref{cor:real022pres}. 
\begin{figure}[htpb]
    \centering
    \includegraphics[width=\linewidth]{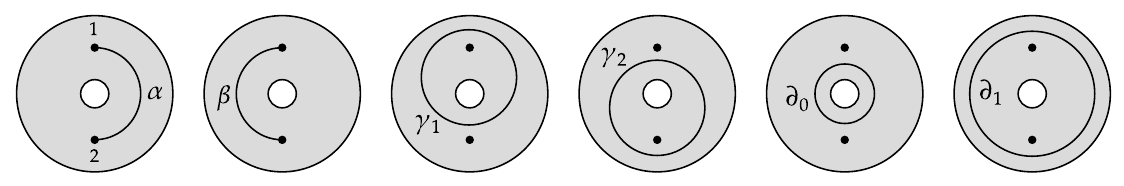}
    \caption{Generators for $\S^{RY}(\Sigma_{0,2,2})$}
    \label{fig:022-gens}
\end{figure}

 \begin{theorem}\label{thm:022pres}
    $\S^{RY}(\Sigma_{0,2,2})$ is the non-commutative algebra generated by $\del_0, \a,\b,\g_1\g_2$ (see Figure \ref{fig:022-gens}) over the commutative ring $R=\Z[A^{\pm1/2},v_i^{\pm1}]$, subject to the following relations.
    \begin{align}
  \b\a&=\a\b-v_1^{-1}v_2^{-1}(A-A^{-1})(\g_2-\g_1) \label{eqn:1} \\
    \g_1\a&=A^2\a\g_1-A(A^2-A^{-2})\b \label{eqn:2} \\
    \g_2\a&=A^{-2}\a\g_2+A^{-1}(A^2-A^{-2})\b. \label{eqn:3} \\
    \g_1\b&=A^{-2}\b\g_1+A^{-1}(A^2-A^{-2})\a. \label{eqn:4} \\
    \g_2\b&=A^2\b\g_2-A(A^2-A^{-2})\a. \label{eqn:5} \\
    \g_1\g_2&=A^{2}(v_1v_2\b^2-2)+(\del_0+\del_1)+(A+A^{-1})^2+A^{-2}(v_1v_2\a^2-2) \label{eqn:6} \\
    \g_2\g_1&=A^{-2}(v_1v_2\b^2-2)+(\del_0+\del_1)+(A+A^{-1})^2+A^{2}(v_1v_2\a^2-2)  \label{eqn:7}\\
  0&=   [\del_0, \del_1]= [\del_0, \a] =  [\del_0, \b] = [\del_0, \g_1]= [\del_0, \g_2]  \label{eqn:9}
  \end{align}
\end{theorem}

\begin{proof}   It is a routine calculation to show that the relations (\ref{eqn:1})-(\ref{eqn:9}) hold in $\S^{RY}(\Sigma_{0,2,2})$. 

Define $\overline{R}=\Z[A^{\pm1/2},v_i^{\pm1},\del_0]$. Consider the alphabet $X=\{\a,\b ,\g_1,\g_2\}$, and let $\langle X\rangle$ be the $\overline{R}$-algebra of finite words in this alphabet. Let $I$ be the $\overline{R}$-ideal in $\langle X\rangle$ generated by relations \eqref{eqn:1}-\eqref{eqn:7}. 

We will apply Bergman's Diamond Lemma \cite{DiamondLemma} to obtain a basis for the algebra $\overline{R}\langle X\rangle/I$. There is a natural surjection $\overline{R}\langle X\rangle/I\onto \SRY(\san)$, which we upgrade to an isomorphism by showing it maps this basis to a basis for $\SRY(\san)$.

To use the Diamond Lemma, we construct a locally confluent terminating reduction system on $\langle X\rangle$. For definitions and more details about this approach, see \cite{DiamondLemma}. For details relevant specifically to skein algebras, see \cite[Section~9]{CookeAskeyWilson}. 

Define a reduction system $S$ on $\overline{R}\langle X\rangle $ by relations \eqref{eqn:1}-\eqref{eqn:7}. Notice all the reduction rules are pairwise, meaning the left hand side is always the product of two letters in $X$. This implies that, to show this system is locally confluent, it is sufficient to show all overlap ambiguities are resolvable.  Here, an overlap ambiguity is $x_1x_2x_3\in \langle X\rangle$ such that $x_1x_2$ and $x_2x_3$ both have reduction rules. It is resolvable if there exist sequences of reductions of $x_1x_2x_3$, beginning with the rules for $x_1x_2$ and $x_2x_3$ respectively that agree at their final expression. There are finitely many overlap ambiguities, so this can be checked by straightforward (if tedious) calculations.

To show this system is terminating, we construct a semigroup partial ordering $\succ$ on $\langle X\rangle$, as in \cite{CookeAskeyWilson}.
First, we establish some notation for $x_{i_1}\cdots x_{i_k}=m\in \langle X\rangle$. Let $|m|=k$, the \textit{length} of $m$. Let $\#(m)$ denote the number of times $\g_1$ or $\g_2$ appears in $m$. Order $X$ by $\a<\b<\g_1<\g_2$ and let the \textit{reduced degree} of $m$ be $|m|$ if there exist $1\leq h<j\leq k$ such that $x_{i_h}>x_{i_j}$, and 0 otherwise. This is a special case of the reduced degree used in \cite{CookeGenus2}. 

Let $m_1,m_2\in \langle X\rangle$ and declare $m_1\succ m_2$ if any of the following is satisfied:
\begin{enumerate}
        \item If $|m_1|>|m_2|$, then $m_1 \succ m_2$.
        \item If $|m_1|=|m_2|$ and the reduced degree of $m_1$ is greater than that of $m_1$, then $m_1\succ m_2$.
        \item If $|m_1|=|m_2|$ and $m_1,m_2$ have the same reduced degree and $\#(m_1)>\#(m_2)$, then $m_1 \succ m_2$.
\end{enumerate}

It is straightforward to check that this is a semigroup partial order and that it is \textit{compatible} with the reduction system defined above, in the sense that the monomials on the right side of equations \eqref{eqn:1}-\eqref{eqn:7} are less than those on the right.
Hence, by the Diamond Lemma, $\overline{R}\langle X\rangle /I$ has an $\overline{R}$-basis given by \textit{irreducible} monomials in $\langle X\rangle$, that is, monomials to which one cannot apply any reduction rules.

We now give an explicit description of this basis. Let $m=x_{i_1}\cdots x_{i_k}$ be an irreducible monomial and recall there is an order on $X$ given by $\a<\b<\g_1<\g_2$. For each pair $(x_1,x_2)$ with $x_1>x_2$ there is a reduction rule, so we must have $y_{i_j}\leq y_{i_{j+1}}$ for each $j$. Hence, $m=\a^{e_1}\b^{e_2}\g_1^{e_3}\g_2^{e_4}$ for some $e_i\in \Z_{\geq0}$. We can apply \eqref{eqn:6} to further reduce the monomial if and only if $e_3e_4\neq 0$. Thus, the set $B=\{\a^{e_1}\b^{e_2}\g_1^{e_3}\g_2^{e_4}\mid e_3e_4=0\}$ is an $\overline{R}$-basis for $\overline{R}\langle X\rangle /I$.

Given that the relations \eqref{eqn:1}-\eqref{eqn:9} are satisfied, there is a natural surjective algebra homomorphism 
\[
\psi:\overline{R}\langle X\rangle /I\onto \S^{RY}(\Sigma_{0,2,2}).
\]
To upgrade $\psi$ to an isomorphism, it remains to show $\psi(B)$ is $\overline{R}$-linearly independent in $\S^{RY}(\Sigma_{0,2,2})$. Our strategy is to quotient $\S^{RY}(\Sigma_{0,2,2})$ even further, and show the image of $\psi(B)$ under this quotient mapping is linearly independent. We use the following reformulation of \cite[Lemma~1.2]{BullockPrz}.

 \begin{lemma}\label{lem:lin-ind-ring}
    Let $S$ be a torsion-free algebra over a commutative, Noetherian integral domain $\mathcal{R}$. Let $J \subset \mathcal{R}$ be a nonzero finitely-generated ideal, let $\pi:S\to S/(J)$ be the natural projection, where $(J)$ is the ideal generated by $J1_S$ in $S$. If $B\subset S$ is finite and $\pi(B)\subset S/(J)$ is $\mathcal{R}/J$-linearly independent, then $B$ is $\mathcal{R}$-linearly independent.
\end{lemma}

In our case, $S$ is $\SRY(\san)$, which is torsion free by \cite[Theorem~2.4]{RogerYang} 
and $\mathcal{R}$ is $\overline{R}$, which is clearly Noetherian. The map $\pi$ is given by the following chain of maps. 
\[
    \S^{RY}(\Sigma_{0,2,2})\to \S^{RY}(\Sigma_{0,1,3})\onto \Z[x,y,z]
\]
The first map is the algebra homomorphism induced by the inclusion $\Sigma_{0,2,2}\into \Sigma_{0,1,3}$. The underlying surfaces are homeomorphic, so the kernel of the induced homomorphism is $(\del_0-A-A^{-1})$. 
By \cite{MoonPuncSph}, $ \S^{RY}(\Sigma_{0,1,3})/(A^{1/2}-1,v_i-1)\cong \Z[x,y,z]$, so the second map is the natural projection for this quotient. 

Consider the ideal $(A^{1/2}-1,\del_1-2,v_i-1)$ of $\overline{R}$ and observe that $\pi$ is the natural projection 
\[
\SRY(\san)\onto \SRY(\san)/(A^{1/2}-1,\del_0-2,v_i-1)\cong \Z[x,y,z]
\]
Hence, by Lemma \ref{lem:lin-ind-ring}, $B$ is $\overline{R}$-linearly independent in $\SRY(\san)$ if and only if $\pi(B)$ is $\Z$-linearly independent in $\Z[x,y,z]$. 
We calculate
\begin{align*}
        \pi(\a)=x,\qquad\quad
        \pi(\b)=yz-x,\qquad\quad
        \pi(\g_1)=y^2-2,\qquad\quad
        \pi(\g_2)=z^2-2.
\end{align*}
so $\pi(B)=\{ x^{e_1}(yz-x)^{e_2}(y^2-2)^{e_3}(z^2-2)^{e_4}\mid e_3e_4=0 \}.$
One can show $\pi(B)$ is $\Z$-linearly independent in $\Z[x,y,z]$ as follows. Suppose
\begin{equation}
        0=\sum_{a,b,c,d}\a_{a,b,c,d} \, x^a(yz-x)^b(y^2-2)^c(z^2-2)^d
    \label{eq:initlincombo}
    \end{equation}
where $\a_{a,b,c,d}\in \Z$ and the subscript runs some finite subset of $\Z_{\geq0}^4$. Evaluate at $x=0$, then show the summands that do not vanish are linearly independent by considering the usual $\Z_{\geq0}^3$-grading on $\Z[x,y,z]$. Factor out a power of $x$ from the initial expression, and apply induction on $e_1$.
\end{proof}

\begin{corollary} \label{cor:orderly}
$\mathcal S^{RY}(\san)$ is orderly finitely generated by simple curves.  
\end{corollary}

\begin{proof} 
It is shown in the proof of Theorem \ref{thm:022pres} that  $B=\{\a^{e_1}\b^{e_2}\g_1^{e_3}\g_2^{e_4}\mid e_3e_4=0\}$ is a basis for $\mathcal S^{RY}(\san)$.  So the simple curves $\a, \b, \g_1, \g_2$ form an orderly finitely generating set.  
\end{proof}

We end this section by providing a simpler presentation for $\S^{RY}(\Sigma_{0,2,2})$.   We use the notation from  Definition~\ref{def:comm}:  $[\alpha, \beta]_A = A \alpha * \beta - A^{-1} \beta * \alpha$.

\begin{corollary}\label{cor:real022pres}
    $\S^{RY}(\Sigma_{0,2,2})$ is the non-commutative $\Z[A^{\pm 1/2},v_i^{\pm1},\del_0,\del_1]$-algebra generated by $\a,\b,\g_1$ subject to the following relations. 
    \begin{align}
        v_1v_2[\b,\a]_A&=(A^{2}-A^{-2})\g_1+(A-A^{-1})(\del_0+\del_1)\label{eqn:tor1}\\
        [\a,\g_1]_A&=(A^2-A^{-2})\b\label{eqn:tor2}\\
        [\g_1,\b]_A&=(A^2-A^{-2})\a\label{eqn:tor3}\\
        v_1v_2A\b\a\g&=v_1v_2A^2\b^2+v_1v_2A^{-2}\a^2+A^2\g^2+A\g(\del_0+\del_1)+\del_0\del_1-(A-A^{-1})^2\label{eqn:tor4}
    \end{align}
\end{corollary}

\begin{proof}
    One can check by hand that these relations are satisfied. It remains to show that they imply the relations in Theorem \ref{thm:022pres}. This can be done by direct computation. It is useful to note that $\g_2=A^{-1}(v_1v_2\a\b-A^{-1}\g_1-\del_0-\del_1)$ and that $\del_1$ is central.
\end{proof}

We will not use the presentation from Theorem \ref{thm:022pres} again. Therefore, from now on, to simplify notation, we will denote $\gamma_1$ by $\gamma$. 

\section{Relationships between skein algebras of small surfaces} \label{sec:relationships}

Let  $\mathcal{S}(\Sigma_{1,0,0})$ denote the usual Kauffman bracket skein algebra of the closed torus.

\begin{theorem}\label{thm:hom-to-torus}
    There exists a surjective algebra homomorphism
    \[
    \phi:\S^{RY}(\Sigma_{0,2,2})\to  \mathcal{S}(\Sigma_{1,0,0})\otimes \Z[A^{\pm 1/2}]
    \]
such that the kernel of $\phi$ is generated over a subset of the center of $\SRY(\san)$.
\end{theorem}

\begin{proof}

Recall the presentation of the closed torus from Theorem 2.1 of  \cite{BullockPrz}:     $\mathcal{S}(\Sigma_{1,0,0})$ is isomorphic to the noncommutative algebra generated by $x_1,x_2,x_3$ over $\Z[A^{\pm 1}]$ subject to the following relations using the commutator from Definition~\ref{def:comm}.
    \begin{align*}
        [x_1,x_2]_A&=(A^2-A^{-2})x_3\\
        [x_3,x_1]_A&=(A^2-A^{-2})x_2\\
        [x_2,x_3]_A&=(A^2-A^{-2})x_1\\
        Ax_1x_2x_3&=A^2x_1^2+A^{-2}x_2^2+A^2x_3^2-2(A^2+A^{-2})
    \end{align*}

    Let $\phi(v_i)=1$, $\phi(\del_0)=-\phi(\del_1)=A+A^{-1}$, $\phi(\b)=x_1, \phi(\a)=x_2,$ and $\phi(\g)=x_3$, and extend $\phi$ linearly and over products. Using Corollary \ref{cor:real022pres}, one checks that $\phi$ is an algebra homomorphism. Surjectivity is immediate.

From the definition, it is easy to see that the kernel contains $v_i-1,\del_0-A-A^{-1},\del_1+A+A^{-1}$.   One can show using the Diamond Lemma (using \eqref{eqn:tor1}-\eqref{eqn:tor3} as reduction rules) and induction that 
\[
\{\b^{e_1}\a^{e_2}\g^{e_3}\mid e_1e_2e_3=0\}
\]
is an $\overline{R}$-basis for $\SRY(\san)$. The image of this basis under $\phi$ is the basis for $\S(\sto)$ constructed in \cite[Thm.~2.1]{BullockPrz}. Hence, $\phi$ is an extension of a ring homomorphism $\overline{R}\to \Z[A^{\pm1/2}]$, so its kernel is the ideal generated in $\SRY(\san)$ by some ideal of $\overline{R}\cdot \emptyset$ (recall $\emptyset$ denotes the empty link), and $\overline{R}\cdot \emptyset$ is the center of $\SRY(\san)$.
\end{proof}

\begin{remark}

Note that $\SRY(\san)/(v_1-1)\cong \mathcal{S}(\sto)\otimes \Z[A^{\pm 1/2}]$, so the above statement could be rephrased as a map between Roger-Yang skein algebras.  

\end{remark}

The homomorphism $\phi$ is not the only instance of a relationship between the Roger-Yang skein algebra and the usual skein algebra observed by the authors. For example, the skein algebras of the thrice punctured disk and the torus with one boundary component are nearly isomorphic, as follows.

\begin{proposition}
    There is a surjective homomorphism  $\SRY(\Sigma_{0,1,3})\to \S(\Sigma_{1,1,0})$ given by identification of generators and sending $A^{1/2}\mapsto A$ and $v_i\mapsto1$. This induces an isomorphism, $ \S(\Sigma_{1,1,0})\cong \SRY(\Sigma_{0,1,3})/(v_i-1)$.
\end{proposition}
\begin{proof}
	In \cite[Theorem~1.1]{MoonPuncSph}, the authors prove that $\SRY(\Sigma_{0,1,3})$ is the $\Z[A^{\pm1/2},v_i^{\pm1}]$-algebra generated by $x_1,x_2,x_3$ subject to the relation $v_i[x_i,x_{i+1}]_{A^{1/2}}=(A-A^{-1})x_{i+2}$ with subscripts taken mod 3. In \cite[Theorem~2.1]{BullockPrz}, the authors show $\S(\Sigma_{1,1,0})$ is the $\Z[A^{\pm1}]$-algebra generated by $x_1,x_2,x_3$ subject to the relation $ [x_i,x_{i+1}]_{A}=(A^2-A^{-2})x_{i+2}$ with subscripts taken mod 3. From this, it is clear that the map described in the statement is a homomorphism with kernel $v_i-1$.
\end{proof}

\begin{remark}
From this proposition, we easily obtain some algebraic information about $\SRY(\Sigma_{0,1,3})/(v_1-1)$. Composing finite-dimensional, irreducible representations of $\S(\Sigma_{1,1,0})$ with this homomorphism yields such representations for $\SRY(\Sigma_{0,1,3})$. Furthermore, note that $\SRY(\Sigma_{0,1,3})/(v_1-1)$ is the specialization of the Roger-Yang skein algebra considered in \cite{KaruoPositive}. As $\S(\Sigma_{1,1,0})$ has a positive basis by \cite{BousseauPositive,MandelQinPositive}, this isomorphism implies $\SRY(\Sigma_{0,1,3})/(v_i-1)$ has a positive basis. It would be interesting to see if this basis agrees with the bracelets basis proposed in \cite{KaruoPositive}.
\end{remark}

\section{Representation theory of $\S^{RY}(\Sigma_{0,2,2})$}  \label{sec:reps}

In this section, we consider $\SRY(\Sigma)$ as a $\C$-algebra instead of a $\Z$-algebra.

\subsection{  $\S^{RY}(\Sigma_{0,2,2})$ is almost Azumaya}

We  refer the interested reader to \cite{BrownGoodearl, FKBLUnicity} for definitions and details about almost Azumaya algebras.

 For $\mathcal S^{RY}(\san)$, the following proof was suggested by Thang Lê (see Remark 5.6 in \cite{KaruoMoonWongCenter}).  





\begin{theorem}
When $A^{2N}=1$, $\mathcal S^{RY}(\san)$ is almost Azumaya.  
\end{theorem} 
\begin{proof}
To show that an algebra is almost Azumaya, we use Theorem 3.6 from \cite{FKBLUnicity}, which shows an algebra is almost Azumaya if it satisifes three conditions:  it is finitely generated as a $\C$-algebra,  it is prime, and  it is finitely generated over its center.   

For $\mathcal S^{RY}(\san)$, the first two conditions follow from \cite{BobbPeiferFinGen, MoonWongComp, BKLQTrace}.   
To prove that $\mathcal S^{RY}(\san)$ is finitely generated over its center, recall that Corollary \ref{cor:orderly} shows $\mathcal S^{RY}(\san)$ is orderly finitely generated by simple loops and arcs $B=\{\a^{e_1}\b^{e_2}\g_1^{e_3}\g_2^{e_4}\mid e_3e_4=0\}$.  Thus application of the Chebyshev polynomial from Proposition \ref{prop:center} produces central elements $T_N(\alpha)$, $T_N(\beta)$, $T_N(\sqrt{v_1v_2} \gamma_1)$, and $ T_N(\sqrt{v_1v_2} \gamma_2)$.   Note that the vertex classes $v_1$ and $v_2$ are scalars and hence commutative.  We can thus decrease the degree on $\a, \b, \g_1$ and $\g_2$ until it is equal to or less than $N$,  the degree of $T_N$.  Thus $\mathcal S^{RY}(\san)$ is finite dimensional over its center,  generated by $\{ \a^{e_1}\b^{e_2}\g_1^{e_3}\g_2^{e_4}\mid   0 \le e_i < N, \; e_3e_4=0\}$ as a $Z(\mathcal S^{RY}(\san))$-module. 

 \end{proof}

We thus have implications for the representation theory for $\mathcal S^{RY}(\san)$, analogous to the Unicity Theorem for the usual Kauffman bracket skein algebra.  For an algebra $\mathcal{A}$, let   $\mathrm{Irrep}(\mathcal{A})$ be the set of all equivalence classes of finite-dimensional irreducible representation of $\mathcal{A}$ and let $\mathrm{MaxSpec}(Z(\mathcal{A}))$ be the set of all maximal ideals of its center $Z(\mathcal{A})$.   For every finite-dimensional irreducible representation $\rho$ of $\mathcal{A}$, one may define its associated central character by using Schur's Lemma to associate a scalar to each element of the center $Z(A)$.  Up to isomorphism, the central character is in bijection with the maximal ideals of $Z(\mathcal{A})$.  Thus, there is a well-defined map  $\chi: \mathrm{Irrep}(\mathcal{A}) \to  \mathrm{MaxSpec}\;(\mathcal{A})$. 


\begin{corollary} 
When $\mathcal{A} = \mathcal S^{RY}(\san)$, and $D$ is the dimension of  $\mathcal{A}$ as a $Z(\mathcal{A})$-module,  every irreducible representation of  $\mathcal S^{RY}(\san)$ has dimension at most $\sqrt{D}$, and   $\chi: \mathrm{Irrep}(\mathcal{A})  \to  \mathrm{MaxSpec}\; Z(\mathcal{A})$ is surjective.  Moreover, there exists a Zariski  open dense subset $U \subset \mathrm{MaxSpec}\; Z(\mathcal{A})$ on which  $\chi^{-1}$ is injective, and the dimension of any irreducible representation in $\chi^{-1}(U)$ is equal to  $\sqrt{D}$. 
\end{corollary}


In the next section, we provide explicit computations on the Zariski open dense subset $U$ and show that the associated representations have dimension $\sqrt{D} = N$ when $A^{2N}=1$ and $N$ is odd.

\subsection{Explicit representations of $\S^{RY}(\Sigma_{0,2,2})$}


Throughout this section, let $A$ be a primitive $2N$th root of unity with $N$ odd. Let $\rho$ be an irreducible representation of $\SRY(\san)$ over a $\C$-vector space $V$. By Schur's Lemma, we can associate a scalar to each element of the center $Z(\SRY(\san) )$. Let $\chi_\rho: Z(\SRY(\san) ) \to \C$ denote the corresponding central character.  We evaluate $\chi_p$ at central elements corresponding to the generators to obtain the 5-tuple  $(t_1, t_2, t_3, d_0, d_1) \in \C^5$ defined by  
\[
    \rho \left(T_N(\sqrt{v_1v_2}\beta) \right) =t_1\id_V, \quad  
    \rho \left(T_N(\sqrt{v_1v_2}\alpha) \right) =t_2\id_V, \quad  
    \rho \left(T_N(\gamma)\right)=t_3\id_V
   \]
    \[
    \rho(\del_0) = d_0\id_V,\quad
    \rho(\del_0) = d_1\id_V. 
\]

Following the skein theory literature, we will call $(t_1, t_2, t_3, d_0, d_1)$ the \emph{classical shadow data} of $\rho$.  The numbers $d_0$ and $d_1$ are also sometimes referred to as the \emph{boundary invariants} of $\rho$.

For the remainder of this section we explore to what extent we can recover $\rho$ from its classical shadow data.    We adapt the method of \cite{TakenovReps} for the closed torus and \cite{HavlicekPosta} for $U_q(so_3)$  to show that when the classical shadow data satisfy some polynomial conditions,  there is a unique finite-dimensional irreducible representation of $\SRY(\san)$ with that classical shadow data.

 \begin{theorem}\label{thm:reps-simple}
 Let $A$ be a primitive $2N$th root of unity for $N$ odd. Let $(t_1,t_2,t_3,d_0, d_1) \in \C^5$ such that  
 \[ t_3 \neq \pm 2, \quad 
t_1^2+t_2^2+t_1t_2t_3\neq 0, \quad 
  T_N(2-d_0^2)= 2 -t_1^2-t_2^2-t_3^2-t_1t_2t_3, \quad 
 d_0 + d_1 = 0.   \] 
 Then there exists a unique irreducible finite-dimensional representation of $\SRY(\san)$ with classical shadow data $(t_1,t_2,t_3, d_0, d_1)$. 
 \end{theorem}

In fact, we will obtain an explicit description of the representation $\rho$. We decompose this construction into a series of lemmas, which follow the approach of \cite{TakenovReps} to the representations of $\S(\sto)$. 

Throughout, we assume that $\rho$ is not the zero representation. Using the presentation of Corollary \ref{cor:real022pres}, one can show if $\rho(\gamma)= 0$ then $\rho$ is the zero representation, so we may assume $\rho(\gamma)\neq 0$.

We will show a slightly more general version of Theorem \ref{thm:reps-simple}. Let $(t_1, t_2, t_3, d_0, d_1)$ be the classical shadow data of $\rho$ and assume it satisfies $t_3\neq 2$ and 
\begin{equation}\label{eqn:ugly-assumption}
0\neq \prod_{k=1}^N (2+d_0d_1+x^2A^{4k+2}+x^{-2}A^{-4k-2}+(d_0+d_1)(xA^{2k}-x^{-1}A^{-2k}(A+A^{-1}-1))).
\end{equation}

Fix a choice of $x\in \C$ such that $x^N+x^{-N}=t_3.$ By Lemma 4 of \cite{TakenovReps}, every eigenvalue of $\rho(\gamma)$ can be written in the form  $\lambda=xA^{2k}+x^{-1}A^{-2k}$ for some $k\in \{1,\ldots,N\}$. Thus let $\lambda_k=xA^{2k}+x^{-1}A^{-2k}$ denote the (potential) eigenvalues of $\rho(\gamma)$, so that  $T_N(\lambda_k)=t_3$ for all $k$. By assumption $t_3\neq \pm 2$, with implies all the $\lambda_k$ are distinct.

Define $V_k=\{\v\in V\mid \rho(\gamma)\v=\lambda_k\v\}$.  Consider the maps
\[    
U_k=A\rho(\b)-xA^{2k}\rho(\a), \quad D_k=A\rho(\b)-x^{-1}A^{-2k}\rho(\a)
\]

The following lemma explains the notation, $U$ for ``up'' and $D$ for ``down.''

\begin{lemma}
    The operators $U_k$ and $D_k$ satisfy $U_k:V_k\to V_{k+1}$ and $D_k:V_k\to V_{k-1}$. Furthermore, for any $\v\in V_k$,
    \begin{align*}
        \rho(\beta)\v&=-\frac{x^{-1}A^{-2k-1}}{xA^{2k}-x^{-1}A^{-2k}}U_k\v+\frac{xA^{2k-1}}{xA^{2k}-x^{-1}A^{-2k}}D_k\v\\
        \rho(\alpha)\v&=-\frac{1}{xA^{2k}-x^{-1}A^{-2k}}U_k\v+\frac{1}{xA^{2k}-x^{-1}A^{-2k}}D_k\v\\
        \rho(\g)\v&=(xA^{2k}+x^{-1}A^{-2k})\v
    \end{align*}
\end{lemma}
\begin{proof}
    This is a straightforward calculation. It is given in Lemmas 8 and 9 of \cite{TakenovReps}.
\end{proof}

\begin{lemma}\label{lem:DU-central}
    $D_{k+1}U_k:V_k\to V_k$ is a homothety.
\end{lemma}
\begin{proof}
For notation, let $P_k=2+d_0d_1+(d_0+d_1)(\lambda_k-x^{-1}A^{-2k}(A+A^{-1}))$.
    \begin{align*}
        D_{k+1}U_k\v&=(A^2\rho(\beta)^2+A^{-2}\rho(\alpha)^2-x^{-1}A^{-2k-1}\rho(\alpha)\rho(\beta)-xA^{2k+1}\rho(\beta)\rho(\alpha))\v &\text{by expanding the definition}\\
        &=[A^2\rho(\beta)^2+A^{-2}\rho(\alpha)^2-A(xA^{2k}+x^{-1}A^{-2k})\rho(\beta)\rho(\alpha)]\v &\text{by applying \eqref{eqn:1}}\\
        &\qquad +(v_1v_2)^{-1} [x^{-1}A^{-2k}(g(A^2-A^{-2})+(d_0+d_1)(A-A^{-1}))]\v&\\
        &=(v_1v_2)^{-1}[-A^2g^2+g(x^{-1}A^{-2k}(A^2-A^{-2})-d_0-d_1)]\v&\text{by applying \eqref{eqn:tor4}}\\
        &\qquad+(v_1v_2)^{-1} [x^{-1}A^{-2k}(d_0+d_1)(A-A^{-1})-d_0d_1+(A-A^{-1})^2]\v&\\
        &=-(v_1v_2)^{-1}[P_k+(x^2A^{4k+2}+x^{-2}A^{-4k-2})]\v &\text{by applying }\rho(\g)\v=\lambda_k\v
    \end{align*}
\end{proof}

\begin{definition}
	Let $E_k:=-(v_1v_2)^{-1}[P_k+(x^2A^{4k+2}+x^{-2}A^{-4k-2})]$ and let $E=\prod_{h=1}^N E_h$. \end{definition}

Note $E$ is the left-hand side of \eqref{eqn:ugly-assumption}, so, by assumption, $E\neq 0$. To simplify notation, we let $U_{N+k}=U_k$ and similarly for $D_k, P_k, E_k$ and $\lambda_k$.

\begin{lemma}\label{lem:UD-nonzero}
The map $U_{k+N-1}U_{k+N-2}\cdots U_{k}:V_k\to V_k$ is nonzero for all $k$.
\end{lemma}
\begin{proof}
    We will show the map $D_{k+1}D_{k+2}\cdots D_{k+N}U_{k+N-1}U_{k+N-2}\cdots U_{k}$ is nonzero. Let $\v\in V_k$. For each $j$, $U_{k+N-j}\cdots U_{k+1}U_{k}\v\in V_j$. Repeatedly applying Lemma \ref{lem:DU-central}, we obtain,
\[
D_{k+1}D_{k+2}\cdots D_{k+N}U_{k+N-1}U_{k+N-2}\cdots U_{k}\v=\left(\prod_{j=1}^N E_{k+j}\right)\v=\left(\prod_{j=1}^N E_{j}\right)\v=E\v
\]
As $E\neq 0$ by assumption, this is not the zero map, so $U_{k+N-1}U_{k+N-2}\cdots U_{k}$ must not be the zero map.
\end{proof}

\begin{lemma}\label{lem:eigenspace}
The space $V$ is $N$ dimensional and admits a basis $\{\v_1,\ldots,\v_N\}$ where each $\v_k$ generates the eigenspace $V_k$.
\end{lemma}
\begin{proof}
 By assumption, $\rho(\g)\neq 0$, so one of its eigenspaces, say $V_{k_0}$, is nonzero. The map $U_{k_0+N-1}\cdots U_{k_0}:V_{k_0}\to V_{k_0}$ is nonzero by Lemma \ref{lem:UD-nonzero}. Therefore, it has an eigenvalue $u\neq 0$ with associated eigenvector $\v_{k_0}$. Set $\v_1=U_NU_{N-1}\cdots U_{k_0}\v_{k_0}$. By construction, $U_{k_0+N-1}\cdots U_{k_0}\v_{k_0}=u\v_{k_0}$, so multiplying by $U_NU_{N-1}\cdots U_{k_0}$ we obtain
    \begin{align*}
        U_NU_{N-1}\cdots U_{k_0}U_{k_0+N-1}\cdots U_{k_0}\v_{k_0}=uU_NU_{N-1}\cdots U_{k_0}\v_{k_0}
        \implies 
        U_NU_{N-1}\cdots U_2U_1\v_1=u\v_1
    \end{align*}
    Now, set $\v_k=U_{k-1}U_{k-2}\cdots U_2U_1\v_1$ for each $k\in \{1,\ldots,N\}$. Let $W$ be a the subspace of $V$ spanned by $\{\v_1,\ldots,\v_k\}$. We will show this subspace is invariant under the action of $\rho(\alpha),\rho(\beta),$ and $\rho(\g)$. First note that it follows from definition and Lemma \ref{lem:DU-central} that
    \[
    U_k\v_k=\begin{cases}
        \v_{k+1}&1\leq k\leq N-1\\
        u\v_1&k=N
    \end{cases}
    \qquad\text{and}\qquad
    D_k\v_k=\begin{cases}
        E_{k-1}\v_{k-1} & 2\leq k \leq N\\
        u^{-1}E_N\v_N & k=1
    \end{cases}.
    \]

    As $E_k$ is a scalar, note that for every $k$, we have $U_k\v_k,D_k\v_k\in W$, so $\rho(\b),\rho(\a),$ and $\rho(\gamma)$ all fix $W$. Hence, as $\rho$ is irreducible, $V=W$. 
\end{proof}

\begin{lemma}\label{lem:calculate-u}
	We have $E=(v_1v_2)^{-N}(t_1^2 +t_2^2 +t_1t_2t_3)$ and $u=-\dfrac{t_1+x^Nt_2}{\sqrt{v_1v_2}^N}.$
\end{lemma}
\begin{proof}
    Define $U,D:W\to W$ for $\v_k\in V_k$ as
\[        U\v_k=-\frac{x^{-1}A^{-2k-1}}{xA^{2k}-x^{-1}A^{-2k}}U_k\v_k:=H^U_k\cdot U_k\v_k
\qquad \text{and}\qquad
        D\v_k=\frac{xA^{2k-1}}{xA^{2k}-x^{-1}A^{-2k}}D_k\v_k:=H_k^D\cdot D_k\v_k
 \]
    and notice $U(V_k)=V_{k+1}$ and $D(V_k)=V_{k-1}$. Also, note $\rho(\beta)=U+D$. Hence,
    \[
    t_1\id_V=T_N(\rho(\beta)\sqrt{v_1v_2})=T_N(\sqrt{v_1v_2}(U+D))
    \]
    The left hand side is a polynomial in $\sqrt{v_1v_2}U$ and $\sqrt{v_1v_2}D$ over $\C[v_1,v_2]$. Consider a monomial of length $m$, where $n$ of the terms are $\sqrt{v_1v_2}U$ and $m-n$ are $\sqrt{v_1v_2}D$. It is a property of $T_N$ that all the monomials are of odd degree, so $m$ is odd. The monomial sends $V_k$ to $V_{k+n-(m-n)}=V_{k+2n-m}$. As $T_N(\sqrt{v_1v_2}(U+D))=t_1\id_V$ is homothety, it must fix all the $V_k$. Hence, the only monomials with nonzero coefficient must satisfy $2n-m\equiv 0\text{ mod }N$. As $n,m\in \{0,\ldots,N\}$, we have $(m,n)$ is $(N,0)$ or $(N,N)$. As $T_N(\sqrt{v_1v_2}(U+D))$ has degree $N$, $T_N(\sqrt{v_1v_2}(U+D))=\sqrt{v_1v_2}^{N}U^N+\sqrt{v_1v_2}^{N}D^N$. 
   
    Let $1\leq k\leq N$. One can calculate
    \begin{align*}
        U^N\v_k&=U^{N-1}(H_k^UU_k\v_k)&\\
        &=H_kH^U_{k+1}\cdots H^U_{k+N-1}U_{k+N-1}\cdots U_k\v_k&\\
        &=H_kH^U_{k+1}\cdots H^U_{k+N-1}u\v_k &\text{from the formula for $U_k\v_k$ in Lemma \ref{lem:eigenspace}}
    \end{align*}
    Also, as $N$ is odd, and $A^2$ is a primitive $N$th root of unity, we have $\prod_{k=1}^N(xA^{2k}-x^{-1}A^{-2k})=x^N-x^{-N}$. Hence,
    \begin{align*}
        H^U_kH^U_{k+1}\cdots H^U_{k+N-1}&=\prod_{k=1}^NH^U_k
        =\prod_{k=1}^N(-\frac{x^{-1}A^{-2k-1}}{xA^{2k}-x^{-1}A^{-2k}})
        =-\frac{x^{-N}A^{-N2k}A^{-N}}{\prod_{k=1}^N(xA^{2k}-x^{-1}A^{-2k})}
        =\frac{x^{-N}}{x^N-x^{-N}}
    \end{align*}
    Thus, $U^N\v_k=\frac{x^{-N}u}{x^N-x^{-N}}\v_k.$
    We can also calculate $D^N\v_k=H_k^DH_{k+1}^D\cdots H_{k-N+1}^DD_{k-N+2}\cdots D_k\v_k$.
    As before,
    \[
    H_k^DH_{k+1}^D\cdots H_{k+N-1}^D=\frac{x^NA^{2kN}A^{-N}}{\prod_{k=1}^N (xA^{2k}-x^{-1}A^{-2k})}
    =-\frac{x^N}{x^N-x^{-N}}
    \]
    and, by applying the formula for $D_k\v_k$ in Lemma \ref{lem:eigenspace},
    \begin{align*}
        \underbrace{D_{k-N+1}D_{k-N+2}\cdots D_{k-1}D_k}_{N \text{ terms}}\v_k
        =u^{-1}E\v_k
    \end{align*}
    Hence, $D^N\v_k=-\frac{u^{-1}x^NE}{x^N-x^{-N}}\v_k$.
    So we have
    \begin{equation}\label{eqn:t1-u-poly}
     t_1\v_k=T_N(\sqrt{v_1v_2}(U+D))\v_k
    =(\sqrt{v_1v_2}^{N}U^N+\sqrt{v_1v_2}^{N}D^N)\v_k
    =\sqrt{v_1v_2}^{N}\frac{x^{-N}u-x^Nu^{-1}E}{x^N-x^{-N}}\v_k
    \end{equation}
    which determines $u$ up to two possibilities.

    To determine $u$ completely, we do the same calculations with $\rho(\a)$. Observe $\rho(\a)=xA^{2k+1}U+x^{-1}A^{1-2k}D$. As before,
    \[
    t_2\id_V=T_N(\rho(\a)\sqrt{v_1v_2})=T_N(\sqrt{v_1v_2}(xA^{2k+1}U+x^{-1}A^{1-2k}D))=-\sqrt{v_1v_2}^{N}x^NU^N-\sqrt{v_1v_2}^{N}x^{-N}D^N.
    \]
    So by the calculations above of $U^N$ and $D^N$,
    \begin{equation}\label{eqn:t2-u-poly}
    t_2=\sqrt{v_1v_2}^{N}\frac{-u+u^{-1}E}{x^N-x^{-N}}.
    \end{equation}
    
    Now, $u$ is a common root of \eqref{eqn:t1-u-poly} and \eqref{eqn:t2-u-poly}. By explicitly finding the roots of these equations and a bit of casework, one can show that they have common roots if and only if $E=(v_1v_2)^{-N}(t_1^2+t_2^2+t_1t_2t_3)$, and, in this case, $-\frac{t_1+x^Nt_2}{\sqrt{v_1v_2}^{N}}$  is always a common root, and it is unique when $u\neq0$, which is the case here.
\end{proof}

We summarize our findings so far in the following proposition. 

\begin{proposition}\label{thm:rho-description}
    Let $\rho$ be a nontrivial, irreducible, finite dimensional representation of $\mathcal{S}^{RY}(\Sigma_{0,2,2})$  such that its classical shadow data $(t_1, t_2, t_3, d_0, d_1)$  satisfy $t_3\neq \pm 2$ and $E\neq 0$. Then necessarily 
    \[
    E=(v_1v_2)^{-N}(t_1^2+t_2^2+t_1t_2t_3)
    \]
 and for any $x$ such that $x^N + x^{-N} = t_3$, $\rho$ is given by
    \begin{align*}
        \rho(\b)\cdot \v_k&=-\frac{x^{-1}A^{-2k-1}}{xA^{2k}-x^{-1}A^{-2k}}U_k\v_k+\frac{xA^{2k-1}}{xA^{2k}-x^{-1}A^{-2k}}D_k\v_k\\
        \rho(\a)\cdot \v_k&=-\frac{1}{xA^{2k}-x^{-1}A^{-2k}}U_k\v_k+\frac{1}{xA^{2k}-x^{-1}A^{-2k}}D_k\v_k\\
        \rho(\g)\cdot \v_k&=(xA^{2k}+x^{-1}A^{-2k})\v_k
    \end{align*}
    where
    \[
    U_k\v_k=\begin{cases}
        \v_{k+1}&1\leq k\leq N-1\\
        u\v_1&k=N
    \end{cases}
    \qquad \text{and}\qquad 
     D_k\v_k=\begin{cases}
        E_{k-1}\v_{k-1} & 2\leq k \leq N\\
        u^{-1}E_N\v_N & k=1
    \end{cases}
    \]
    and $u=-\frac{t_1+x^Nt_2}{\sqrt{v_1v_2}^{N}}$ and 
    \begin{align*}
        E_k&=-v_1^{-1}v_2^{-1}(P_k+(x^2A^{4k+2}+x^{-2}A^{-4k-2}))\\
        P_k&=2+d_0d_1+(d_0+d_1)(\lambda_k-x^{-1}A^{-2k}(A+A^{-1}))
    \end{align*}
\end{proposition}

Observe that  $\rho$ is entirely determined, up to equivalence, by its classical shadow data. 

One may ask if the converse of Proposition \ref{thm:rho-description} is true. Namely, on may ask if  $(t_1,t_2,t_3,d_0,d_1)\in \C^5$ such that $t_3\neq \pm 2$ and $E=(v_1v_2)^{-N}(-t_1^2-t_2^2-t_1t_2t_3)\neq 0$, do the formulas for $\rho(\alpha),\rho(\beta),\rho(\gamma)$ define a representation of $\SRY(\san)$. The following Proposition describes the extent to which the converse fails.

\begin{proposition}
    If $(t_1,t_2,t_3,d_0,d_1)\in \C^5$ satisfies the conditions of Theorem \ref{thm:rho-description}, the formulas in Proposition \ref{thm:rho-description} define a representation if and only if $d_0+d_1=0$.
\end{proposition}
\begin{proof}
	If $d_0+d_1=0$, one can check this by direct computation.
	
Conversely, assume those formulas yield a representation. We can check by hand that 
    \[
    v_1v_2A\rho(\b)\rho(a)\cdot \v_1-  v_1v_2A^{-1}\rho(\a)\rho(\b)\cdot \v_1=(A^2-A^{-2})\rho(\g)\cdot \v_1+(A-A^{-1})(d_0+d_1)\v_1
    \]
    if and only if $d_0+d_1=0$ or
    \begin{equation}\label{eq:rep-cond-1}
        A+A^{-1}=x(1-A^2)+x^{-1}(A^{-3}-A^{-2}+1-A)
    \end{equation}
    Similar calculations show 
    \[
    v_1v_2A\rho(\b)\rho(a)\cdot \v_N-  v_1v_2A^{-1}\rho(\a)\rho(\b)\cdot \v_N=(A^2-A^{-2})\rho(\g)\cdot \v_N+(A-A^{-1})(d_0+d_1)\v_N
    \]
    holds if and only if $d_0+d_1=0$ or 
    \begin{equation}\label{eq:rep-cond-2}
        0=(1-A)x^2+A+A^2+A^3
    \end{equation}
    Equations \ref{eq:rep-cond-1} and \ref{eq:rep-cond-2} have no common solutions $x$, so they cannot hold simultaneously. Hence, we must have $d_0+d_1=0$.
\end{proof}

Now, to complete the proof of Theorem~\ref{thm:reps-simple} it suffices to observe that if $d_0+d_1=0$, one can apply Lemma 4 of \cite{TakenovReps} to show
\begin{equation}\label{eqn:simplifiedE=ts}
 \prod_{k=1}^N E_k=(v_1v_2)^N(t_1^2+t_2^2+t_1t_2t_3) 
    \iff
    T_N(2-d_0^2)-2=-t_1^2-t_2^2-t_3^2-t_1t_2t_3.
\end{equation}

\begin{remark} 
We may think of 
$
d_0+d_1=0$ and $T_N(2-d_0^2)-2=-t_1^2-t_2^2-t_3^2-t_1t_2t_3
$
as analogues of the puncture invariant conditions for the skein algebra of the torus by \cite[Lem.~7]{TakenovReps}. 
\end{remark}

\begin{remark} Observe  that there are alternative approaches to construct representations.  For example, one can construct a representation of $\S^{RY}(\Sigma_{0,2,2})$ by composing the quantum trace map of \cite{BKLQTrace} from $\S^{RY}(\Sigma_{0,2,2})$ into a quantum torus with a representation of the quantum torus.  However, it is unknown whether the resulting representation is irreducible or what are the irreducible components.  

Alternatively, one can compose the map $\phi$ from Theorem~\ref{thm:hom-to-torus} with an  irreducible representation of $\mathcal{S}(\Sigma_{1,0,0})$ \cite{FrohmanGelca, TakenovReps, YuReps} to obtain  an irreducible representation of $\S^{RY}(\Sigma_{0,2,2})$.  Our computations above recover all the representations produced in this way.  

Note that if we substitute $d_0=A+A^{-1}$ and $d_1=-A-A^{-1}$, $P_k$ becomes $-A^2-A^{-2}$, meaning that it corresponds under the homomorphism $\phi: \mathcal{S}^{RY}(\Sigma_{0,2,2})\to \mathcal{S}(\Sigma_{1,0,0})$ to a nullhomotopic loop in the closed torus. Similarly, if we substitute $d_0=A+A^{-1}$ into \eqref{eqn:simplifiedE=ts}, we recover the condition on the puncture invariant for the once punctured torus at $p=-A^2-A^{-2}$, namely, $T_N(-A^2-A^{-2})=-t_1t_2t_3-t_1^2-t_2^2-t_3^3+2$, given in  \cite[Lem.~7]{TakenovReps}.   Our investigation was motivated by whether all of the representations of $\SRY(\san)$ were constructed this way, which we found to the negative.  \\
\end{remark}

\begin{problem}
 Classify the representation theory of $\SRY(\san)$ when the classical shadow data do \emph{not} meet the requirements of Theorem~\ref{thm:reps-simple}.  \end{problem}

\section{Evidence for Positivity Conjecture of $\S^{RY}(\san)$} \label{sec:pos}

We consider a basis for a specialization of $\S^{RY}(\san)$ that was proposed in \cite{KaruoPositive} to be a positive basis for this algebra, meaning that the structure constants for multiplication in that basis are in $\Z_{\geq 0} [ A^{\pm \frac12}]$.   

By the Product-to-Sum formula of \cite{FrohmanGelca}, multiplication in the skein algebra of the closed torus $\S(\sto)$ is understood completely, and we describe a method that uses the  surjective homomorphism of Theorem \ref{thm:hom-to-torus} to deduce the structure constants of the specialization of $\S^{RY}(\san)$.  
We lay the groundwork for this approach in Section \ref{sec:torusknots} and give some closed form formulas for some infinite families of products in Section \ref{sec:strconst}.   We find that these formulas support the positivity conjecture of \cite{KaruoPositive} for  $\S^{RY}(\san)$. 

\subsection{A geometric basis for $S^{RY}(\Sigma_{0,2,2})$}

In this subsection, we construct a geometric basis $\{\twos{n}{k}_g\}$ for $S^{RY}(\Sigma_{0,2,2})$ and investigate some of its basic properties. We will use this basis in the next subsection to define another basis that corresponds to the one proposed in \cite{KaruoPositive}.

\begin{figure}
\centering
    \includegraphics[width=0.9\linewidth]{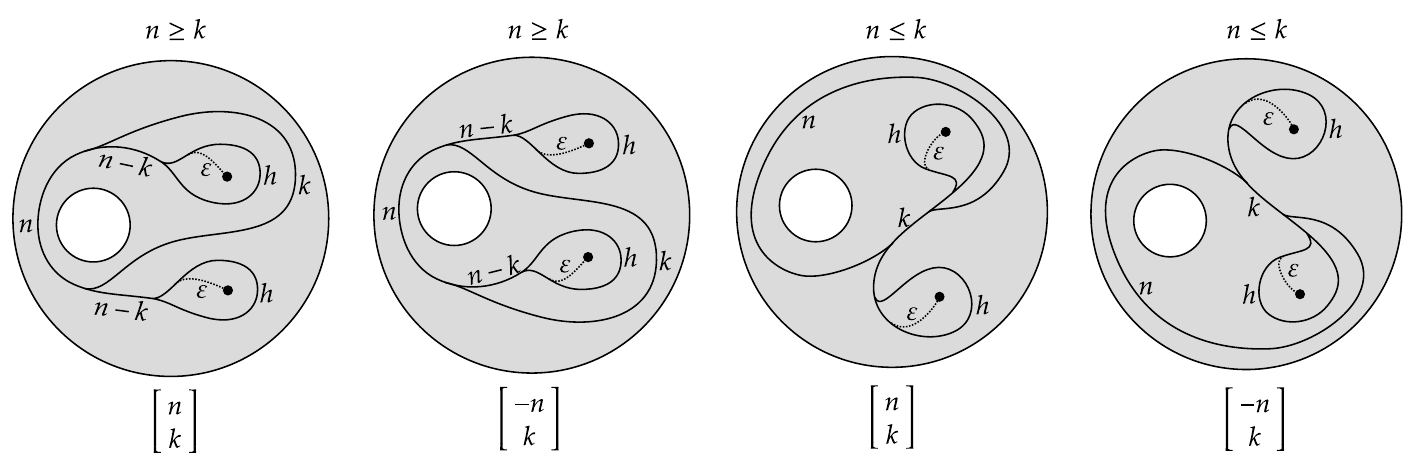}
    \caption{The geometric basis for $\S^{RY}(\Sigma_{0,2,2})$. Here, $\e=n-k\mod 2$ and $h=(n-k+\e)/2$.}
    \label{fig:geometric-basis}
\end{figure}

\begin{definition}
	Define $\twos{n}{k}_g$, the \textit{geometric basis}, by Figure \ref{fig:geometric-basis}.
\end{definition}

One can show by straightforward casework that $\{\twos{n}{k}_g\mid n,k\in \Z\backslash\{0\}\}$ is a basis for $\SRY(\san)$ over $\Z[A^{\pm 1/2}, \del_0,\del_1]$.

\begin{example}
We have \\
\begin{center}
\includegraphics[width=0.7\textwidth]{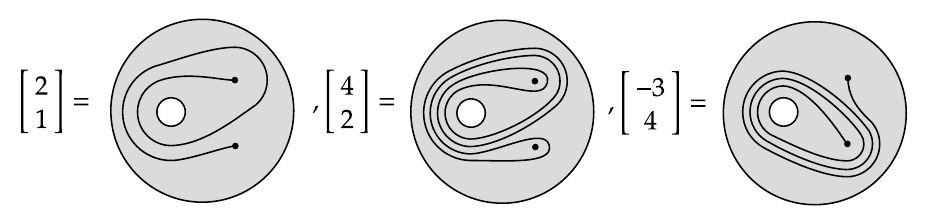}
\end{center}
\end{example}

\begin{remark}
Note that $\twos{n}{k}_g$ is a knot if and only if $ n = k \mod 2$.
\end{remark}

\begin{lemma}\label{lem:sl2z-action}
    Let $\Lambda$ be the subgroup of $\lilmat{a}{b}{c}{d}\in \SL_2\Z$ such that $a$ is odd and $c$ is even (this implies $d$ is also odd). Then there is a well-defined action of $\Lambda$ on curves in $\san$ by
\[
\fourmat{a}{b}{c}{d}\cdot \two{n}{k}_g=\two{c(n+k)/2+dn}{n(2a-c+4b-2d)/2+k(2a-c)/2}_g.
\]
\end{lemma}
\begin{proof}
In this proof, let $(n,k)$ denote the $(n,k)$-link on the four-punctured sphere. The group $SL_2\Z$ acts on $\Sigma_{0,4,0}$ by left multiplication of $(n,k)$ (see \cite{Rhea4puncsphere}). By the natural inclusion of $\Sigma_{0,4,0}\hookrightarrow \san$, we obtain an injective map of loops in $\Sigma_{0,4,0}$ to loops in $\san$. This map can be chosen so that it sends $(\frac{n+k}{2},n)\mapsto \twos{n}{k}_g$, for all $n,k$ of the same parity. From this, one can calculate that left multiplication by an element of $\Lambda$ on $(\frac{n+k}{2},n)$ by left multiplication induces the action given in the statement on $ \twos{n}{k}_g$.

We can extend this action to arcs in $\san$ as follows. If $\omega\subset \san$ is a simple closed curve bounding a twice punctured disk, let $\omega^{arc}$ be the unique arc $\theta\subset \san$ such that $v_1v_2\theta^2-2=\omega$, which exists by Example \ref{ex:T2arc}. 
If $\theta\subset\san$ is an arc and $M\in \Lambda$, define
\[
    M\cdot \theta=(M\cdot (v_1v_2\theta^2-2))^{arc}.
\]
In other words,  the boundary of a regular neighborhood of $\omega^{arc}$ is $\omega$, and the action of $M \in \SL_2\Z$ should respect this relationship. 
One can check the action on arcs defined this way agrees with the one given in the statement. One can also check that $\lilmat{a}{b}{c}{d}\in \SL_2\Z$ maps arcs to arcs and knots to knots via this action if and only if $a$ is odd and $c$ is even.

This extends to an action of $\SL_2\Z$ on $\SRY(\san)$ given by
    \[
    M\cdot (p(A)\cdot \omega)=p(A^{\det M})\cdot (M\cdot \omega)
    \]
where $\omega$ may be a simple knot or arc, and $p(A)\in \Z[A^{\pm1}]$. 

 \end{proof}

\begin{remark}
    The action of $M\in \Lambda$ either fixes or switches $\del_0$ and $\del_1$, so in particular it always fixes $\del_0+\del_1$ and $\del_0\del_1+(A+A^{-1})^2$.
\end{remark}

\begin{figure}
\centering
    \includegraphics[width=0.8\linewidth]{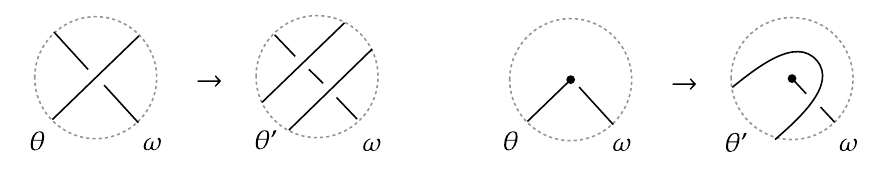}
    \caption{The intersections of $\omega$ and $\theta'$ from the proof of Lemma \ref{lem:intersection-nb}.}
    \label{fig:doubled-crossing-nb}

\end{figure}

\begin{definition}
	Define the geometric intersection number of two curves in $\Sigma_{0,2,2}$ as their minimal number of intersections in the interior of $\Sigma_{0,2,2}$ plus half the number of times they intersect at the punctures.
\end{definition}

If one of the curves is a loop then this is the usual intersection number. If both curves are arcs (each with distinct endpoints), it is $1$ more than their intersection number on the interior of $\Sigma_{0,2,2}$.

\begin{lemma}\label{lem:intersection-nb}
    If $\gcd(n_i,k_i)=1$, the geometric intersection number of $\twos{n_1}{k_1}_g$ and $\twos{n_2}{k_2}_g$ is $|n_1k_2-n_2k_1|$.
\end{lemma}
\begin{proof}
 First, consider the case where $\twos{n_1}{k_1}_g$ and  $\twos{n_2}{k_2}_g$ are both loops.  We represent them as loops in the four-punctured sphere, respectively $(\frac{n_1+ k_1}{2},n_1)$ and $(\frac{n_2+ k_2}{2},n_2)$.   Observe we must have $\gcd(\frac{n_i+ k_i}{2},n_i)=1$, and recall by \cite[Lemma~4.1]{Rhea4puncsphere} (see also Section 2.2.5 of \cite{farbmarg}) that the geometric intersection number of loops in the four-punctured sphere $(N_1,K_1)$ and $(N_2,K_2)$ is $2|N_1K_2-N_2K_1|$ when $\gcd(N_i,K_i)=1$. Applying this when $N_i=\frac{n_i+ k_i}{2}$ and $K_i=n_i$ gives the desired equality.

Now suppose one of the curves, say $\twos{n_1}{k_1}_g:=\theta$ is an arc, and the other $\twos{n_2}{k_2}_g:=\omega$ is a loop, then let $\theta'=v_1v_2\twos{n_1}{k_1}_g^2-2=\twos{2n_1}{2k_1}_g$, which is the boundary of a regular neighborhood of $\theta$, as in Example \ref{ex:T2arc}. Then, from Figure \ref{fig:doubled-crossing-nb} (Left), we see the intersection number of $\omega$ and $\theta'$ is twice that of $\theta$ and $\omega$, and, as $\omega$ and $\theta'$ are loops. One can show that if $\twos{n_1}{k_1}_g$ is an arc then $\gcd(n_1+k_1,2n_1)=1$, so we can apply \cite[Lemma~4.1]{Rhea4puncsphere} as before.

If both curves are arcs, we see from Figure \ref{fig:doubled-crossing-nb} that the intersection number of $\omega$ and $\theta$ is half that of $\omega$ and $\theta'$, and we can apply the previous argument.
\end{proof}

\begin{lemma}\label{lem:easy-prod-sum}
Let $\twos{n_1}{k_1}_g$ and $\twos{n_2}{k_2}_g$ be curves such that $n_1k_2-n_2k_1=\pm 1$. Then
\begin{equation}\label{eqn:easy-prod-sum-ka}
    (v_1v_2)^{\delta_{arcs}}\two{n_1}{k_1}_g*\two{n_2}{k_2}_g=A^{\lildet{n_1}{n_2}{k_1}{k_2}}\two{n_1+n_2}{k_1+k_2}_g+A^{-\lildet{n_1}{n_2}{k_1}{k_2}}\two{n_1-n_2}{k_1-k_2}_g+\delta_{arcs}(\del_0+\del_1).
\end{equation}
where $\delta_{arcs}=1$ if $\twos{n_1}{k_1}$ and $\twos{n_2}{k_2}$ are both arcs and 0 otherwise.
\end{lemma}
\begin{proof}
If $\twos{n_1}{k_1}_g$ if a knot and $\twos{n_2}{k_2}_g$ is an arc (resp. both are arcs) the action of $M_1$ (resp. $M_2$) reduces \eqref{eqn:easy-prod-sum-ka} to an identity that can be computed by hand, where
\[
M_1=\begin{bmatrix}
        n_1&-\frac{n_1+k_1}{2}\\
        -2n_2&n_2+k_2
    \end{bmatrix}
    \qquad
    M_2=\begin{bmatrix}
        n_1-n_2& \frac{1}{2}(n_2+k_2-n_1-k_1)\\
        2n_1&-n_1-k_1
    \end{bmatrix}
\]
The case where $\twos{n_1}{k_1}_g$ is an arc and $\twos{n_2}{k_2}_g$ is a knot follows similarly.
\end{proof}

\subsection{Description of pre-image of torus knots} \label{sec:torusknots}
Let ${n\choose k}$ denote the $(n,k)$-torus link in the closed torus $\sto$. 
In this subsection, we construct a basis $\{\twos{n}{k}\}$ for $\SRY(\san)$ such that $\phi(\twos{n}{k})={n\choose k}$, where $\phi:\SRY(\san)\to \S(\sto)\otimes \Z[A^{\pm 1/2}]$ is the homomorphism defined in Corollary \ref{cor:real022pres}. Our proof uses a similar technique to that of Theorem 1 in \cite{FrohmanGelca}. 

\begin{definition}
	Let $d=gcd(n,k)$ and define $\two{n}{k}=\two{n/d}{k/d}_g^d$.
\end{definition}
 Note if $\twos{n/d}{k/d}_g$ is a knot, we have $\twos{n}{k}_g=\twos{n}{k}$. Also, if $\gcd(n,k)=1$, $\twos{n}{k}_g=\twos{n}{k}$.
\begin{remark}
	This basis is the basis proposed in \cite{KaruoPositive} as a positive basis for the Roger-Yang skein algebra.
\end{remark}

First choose the curves representing $x_1,x_2,x_3$ in the torus so that 

\[
\phi(\b)=x_1={1\choose 0}\quad,\quad \phi(\a)=x_2={0\choose 1}\quad\text{,}\quad\phi(\g)=x_3={1\choose 1}.
\]
Note that if $\twos{n_1}{k_1}$ and $\twos{n_2}{k_2}$ do not intersect then $\twos{n_1}{k_1}*\twos{n_2}{k_2}=\twos{n_1+n_2}{k_1+k_2}$.

\begin{theorem}\label{thm:phi-of-nk}
    If $\gcd(n,k)=1$, $\phi(\twos{n}{k})={n\choose k}$.
\end{theorem}
\begin{proof}
The proof proceeds by induction on $\twos{n}{k}$ with respect to lexicographical order denoted $<$ defined by $\twos{n_1}{k_1}< \twos{n_2}{k_2}$ if $|n_1|< |n_2|$, or $|n_1|=|n_2|$ and $|k_1|<|k_2|$. For the base cases, we have by the definition of the $\phi$ that
\begin{align*}
    \phi(\a)=\phi(\twos{0}{1})={0\choose 1},\quad \phi(\g)=\phi(\twos{1}{0})={1\choose 0},\quad \phi(\g)=\phi(\twos{1}{1})={1\choose 1}.
\end{align*}
Also, one can calculate $\phi(\twos{2}{1})={2\choose 1}$ and $\phi(\twos{1}{2})={1\choose 2}$.

Let $N,K\in \Z$ be relatively prime and assume that for all $\twos{n}{k}\leq \twos{N}{K}$ if $\gcd(n,k)=1$ then $\phi(\twos{n}{k})={n\choose k}$.

We consider only the case where $0<K<N$, as the other case is similar. By \cite[Lemma~1]{FrohmanGelca}, for any $(N,K)$ such that $N\geq 3$, $0<K<N$ and $\gcd(N,K)=1$, there exist $u,v,w,z$ such that $uz-vw=\pm 1$, $u+w=N$, $v+z=K$, $w\in \{1,\ldots,N-1\}$, $u\in \{1,\ldots,N-2\}$, and $v,z\in \Z_{>0}$. It follows that the pairs $(u,v),(w,z),$ and $(u-w,z-v)$ are each relatively prime.

As $uz-vw=\pm 1$, it follows from Lemma \ref{lem:intersection-nb} that one of $\twos{u}{v}$ and $\twos{w}{z}$ must be an arc. If exactly one is an arc, by Lemma \ref{lem:easy-prod-sum},
\[
\two{u}{v}*\two{w}{z}=A^{\lildet{u}{w}{v}{z}}\two{N}{K}+A^{-\lildet{u}{w}{v}{z}}\two{u-w}{z-v}
\]
Because $(u,v),(w,z),$ and $(u-w,z-v)$ are each relatively prime, and $|u|,|w|,|u-w|<|N|$, we can apply the induction hypothesis, which tells us that when we apply $\phi$ to the above equation, we get
\[
{u\choose v}*{w\choose z}=A^{\lildet{u}{w}{v}{z}}\phi\left(\two{N}{K}\right)+A^{-\lildet{u}{w}{v}{z}}{u-w\choose z-v}.
\]
The product-to-sum on the torus from \cite{FrohmanGelca} tells us ${u\choose v}*{w\choose z}=A^{\lildet{u}{w}{v}{z}}{N\choose K}+A^{-\lildet{u}{w}{v}{z}}{u-w\choose z-v}$. Hence, $\phi(\twos{N}{K})={N\choose K}$.

Now we consider the case where both $\twos{u}{v}$ and $\twos{w}{z}$ are arcs. By Lemma \ref{lem:easy-prod-sum},
\[
v_1v_2\two{u}{v}*\two{w}{z}=A^{\lildet{u}{w}{v}{z}}\two{N}{K}+A^{-\lildet{u}{w}{v}{z}}\two{u-w}{v-z}+\del_0+\del_1.
\]
As before, we can apply the induction hypothesis, which tells us that when we apply $\phi$ to the equation above, we get
\[
{u\choose v}*{w\choose z}=A^{\lildet{u}{w}{v}{z}}\phi\left(\two{N}{K}\right)+A^{-\lildet{u}{w}{v}{z}}{u-w\choose z-v}
\]
which, by the product-to-sum on the torus, means $\phi(\twos{N}{K})={N\choose K}$.
\end{proof}

\begin{corollary}
    If $\gcd(n,k)=d$, $\phi(\twos{n}{k})=\phi(\twos{n/d}{k/d})^d={n\choose k}$
\end{corollary}

\subsection{Structural constants for infinite families of products}  \label{sec:strconst}

In this subsection, we make progress towards a product-to-sum formula for $\SRY(\san)$, including giving closed formulas for products of certain infinite families of curves with the arc $\twos{0}{1}$. These formulas involve coefficients that are positive multiples of quantum integers and hence provide evidence that the bracelets basis proposed in \cite{KaruoPositive} is a positive basis for $\SRY(\san)$. To determine these formulas, we use the homomorphism $\phi:\SRY(\san)\to \S(\sto)$ to ``pull back" the product-to-sum on the torus given by \cite{FrohmanGelca}, following the arguments in \cite{SikiFast}.

Throughout this subsection, we consider the specialization of the Roger-Yang skein algebra from \cite{KaruoPositive}, given by setting $v_1=v_2=1$. Note that $\phi$ of Theorem \ref{thm:hom-to-torus} descends to a well-defined homomorphism on this specialization.

\begin{definition}
    Let $T_n$ be the $n$th Chebyshev and define
    \[
    \two{n}{k}_T=T_{\gcd(n,k)}\two{\frac{n}{{\gcd(n,k)}}}{\frac{k}{\gcd(n,k)}}
    \]
\end{definition}

It is important to note that, due to the definition of $\twos{n}{k}$, if $\twos{n}{k}^{arc}$ exists (that is, if $n/2$ and $k/2$ are integers with different parity) then 
\[
\two{n}{k}_T=\two{n/2}{k/2}^2-2=v_1^{-1}v_2^{-1}\two{n}{k}+2(v_1^{-1}v_2^{-1}-1)=\two{n}{k}
\]
because, again, we set $v_1=v_2=1$.

\begin{definition}[\protect{\cite[Def.~3.5]{SikiFast}}]
    The \textit{Frohman-Gelca discrepancy} of $\twos{n_1}{k_1}_T* \twos{n_2}{k_2}_T$ is
    \[
    D\begin{bmatrix}
        n_1&n_2\\
        k_1&k_2
    \end{bmatrix}=\two{n_1}{k_1}_T* \two{n_2}{k_2}_T-A^{\left|\begin{smallmatrix}
        n_1&n_2\\
        k_1&k_2
    \end{smallmatrix}\right|}\two{n_1+n_2}{k_1+k_2}_T-A^{-\left|\begin{smallmatrix}
        n_1&n_2\\
        k_1&k_2
    \end{smallmatrix}\right|}\two{n_1-n_2}{k_1-k_2}_T
    \]
\end{definition}

\begin{theorem}[\protect{\cite[Thm.~3.1]{SikiFast}}]
    For $n_i,k_i\in \Z$,
    \begin{align*}
    &\two{n_1}{k_1}_T*D\begin{bmatrix}
        n_2&n_3\\
        k_2&k_3
    \end{bmatrix}+A^{\left|\begin{smallmatrix}
        n_2&n_3\\
        k_2&k_3
    \end{smallmatrix}\right|}D\begin{bmatrix}
        n_1&n_2+n_3\\
        k_1&k_2+k_3
    \end{bmatrix}+A^{-\left|\begin{smallmatrix}
        n_2&n_3\\
        k_2&k_3
    \end{smallmatrix}\right|}D\begin{bmatrix}
        n_1&n_2-n_3\\
        k_1&k_2-k_3
    \end{bmatrix}\\
    &\qquad\qquad=D\begin{bmatrix}
        n_1&n_2\\
        k_1&k_2
    \end{bmatrix}*\two{n_3}{k_3}_T+A^{\left|\begin{smallmatrix}
        n_1&n_2\\
        k_1&k_2
    \end{smallmatrix}\right|}D\begin{bmatrix}
        n_1+n_2&n_3\\
        k_1+k_2&k_3
    \end{bmatrix}+A^{-\left|\begin{smallmatrix}
        n_1&n_2\\
        k_1&k_2
    \end{smallmatrix}\right|}D\begin{bmatrix}
        n_1-n_2&n_3\\
        k_1-k_2&k_3
    \end{bmatrix}.
    \end{align*}
\end{theorem}

We will also use the following corollary.
\begin{corollary}[\protect{\cite[Cor.~3.3]{SikiFast}}]\label{cor:sikiwang-3.3}
    For $p,q\in \Z$,
    \begin{align*}
        D\begin{bmatrix}
        p+1&0\\
        q&1
    \end{bmatrix}
    &=A^{-q}\two{1}{0}*D\begin{bmatrix}
        p&0\\
        q&1
    \end{bmatrix}-A^{-2q}\begin{bmatrix}
        p-1&0\\
        q&1
    \end{bmatrix}\\
    &\quad +A^{-p-q}D\begin{bmatrix}
        1&p\\
        0&q-1
    \end{bmatrix}-A^{-q}D\begin{bmatrix}
        1&p\\
        0&q
    \end{bmatrix}*\two{0}{1}+A^{p-q}D\begin{bmatrix}
        1&p\\
        0&q+1
    \end{bmatrix}
    \end{align*}
\end{corollary}

\begin{lemma}\label{lem:easy-discrep}
    If $|n_1k_2-k_1n_2|=0$, $D\fourmat{n_1}{k_1}{n_2}{k_2}=0$. If $|n_1k_2-k_1n_2|=1$,
    \[
    D\lilmat{n_1}{k_1}{n_2}{k_2}=\begin{cases}
        \del_0+\del_1&\text{both }\twos{n_1}{k_1},\twos{n_2}{k_2}\text{ are arcs}\\
        0&\text{otherwise}
    \end{cases}
    \]
\end{lemma}
\begin{proof}
    This is an immediate consequence of Lemma \ref{lem:easy-prod-sum}.
\end{proof}

\begin{definition}
    Let $[n]_q=\frac{q^{n}-q^{-n}}{q-q^{-1}}=q^{1-n}+q^{3-n}+\cdots+q^{n-3} +q^{n-1}$ be a quantum integer.
\end{definition}

\begin{definition}
    To simplify the formulas, we use normalized Chebyshev polynomials $\overline{T}_n(x)$, defined by $\overline{T}_0(x)=1$ and for all $n\geq 1$, $\overline{T}_n(x)=T_n(x)$.
\end{definition}

\begin{proposition}\label{prop:discrep-p001}
Let $p\geq 1$ and $\epsilon\in \{0,1\}$ such that $\epsilon\equiv p\mod 2$.
    \[
    D\fourmat{p+1}{0}{0}{1}=(\del_0+\del_1)(\overline{T}_p(\two{1}{0})+[3]_A\cdot \overline{T}_{p-2}(\two{1}{0})+[5]_A\cdot \overline{T}_{p-4}(\two{1}{0})+\cdots+[2\lfloor p/2\rfloor+1]_A\cdot \overline{T}_{\epsilon}(\two{1}{0}))
    \]
\end{proposition}
\begin{proof}
    By Corollary \ref{cor:sikiwang-3.3}, applying Lemma \ref{lem:easy-discrep},
    \begin{align*}
        D\fourmat{p+1}{0}{0}{1}&=\two{1}{0}*D\fourmat{p}{0}{0}{1}-D\fourmat{p-1}{0}{0}{1}+A^{-p}D\fourmat{1}{p}{0}{-1}-D\fourmat{1}{p}{0}{0}*\two{0}{1}+A^pD\fourmat{1}{p}{0}{1}\\
        &=\two{1}{0}*D\fourmat{p}{0}{0}{1}-D\fourmat{p-1}{0}{0}{1}+\delta_p(A^{-p}+A^p)(\del_0+\del_1)\\
    \end{align*}
    where $\delta_p=0$ is $p$ is odd and $1$ is $p$ is even.
    Using this recurrence relation, one can show the claim by induction on $p$. The following property of the Chebyshevs is useful.
    \begin{equation}\label{eq:chev-identity}
        \twos{1}{0}*T_k(\twos{1}{0})=T_{k+1}(\twos{1}{0})+T_{k-1}(\twos{1}{0}).
    \end{equation}
    The base cases $p=1,2$ can be verified by hand.
\end{proof}

\begin{corollary}
    Let $N_1,N_2,K_2\in \Z_{\geq 0}$ such that $N_1+K_1$ and $N_2+K_2$ are odd and $1=K_2N_1-N_2K_1$. Then for all $p\geq 1$,
    \begin{align*}
        T_p(\two{N_1}{K_1})*\two{N_2}{K_2}
        =
        A^{p}\two{pN_1+N_2}{pK_1+K_2}_T
        +A^{-p}\two{pN_1-N_2}{pK_1-K_2}_T
        +
        (\del_0+\del_1)(\sum_{k=0}^{\lfloor p/2\rfloor}[2k+1]_A\cdot T_{p-2k}(\two{N_1}{K_1}))
    \end{align*}
\end{corollary}
\begin{proof}
Recall the action of $\Lambda$ on $\SRY(\san)$ given in Lemma \ref{lem:sl2z-action}. Consider the action of
\[
\begin{bmatrix}
    N_2+K_2&\frac{1}{2}(N_1+K_1-N_2-K_2)\\
    2N_2&N_1-N_2
\end{bmatrix}\in \Lambda
\]
on each term of the formula obtained in Proposition \ref{prop:discrep-p001}.
\end{proof}

\begin{corollary}
    For all $z\geq 0$,
    \begin{align*}
        T_p(\two{1}{2z})*\two{0}{1}
        =
        A^p\two{p}{2pz+1}
        +A^{-p}\two{p}{2pz-1}
        +
        (\del_0+\del_1)(\sum_{k=0}^{\lfloor p/2\rfloor}[2k+1]_A\cdot T_{p-2k}(\two{1}{2z})).
    \end{align*}
\end{corollary}

We can also calculate $D\lilmat{p+1}{0}{1}{1}$. First, we need a lemma.

\begin{lemma}\label{lem:discrep-1p02}
Let $p\geq 1$. Take representatives for the elements of $\Z/4\Z$ in the set $\{-1,0,1,2\}$. Let $[p]$ be the representative of $p$.
    \[
    D\fourmat{1}{p}{0}{2}=(1-\delta_{[p]}^2)(\del_0+\del_1)A^{[p]}\two{\frac{p+[p]}{2}}{1}+(\delta_{[p]}^{-1}+\delta_{[p]}^1)(\del_0\del_1+(A+A^{-1})^2)
    \]
\end{lemma}
\begin{proof}
    We will prove four cases (depending on the residue of $p\mod 4$) separately. Explicitly, we show the following formulas hold.
    \[
    v_1v_2\two{1}{0}*\two{p}{2}=
    \begin{cases}
       A^2\twos{p+1}{2}+A^{-2}\twos{p-1}{2}+A^{-1}(\del_0+\del_1)\twos{\frac{p-1}{2}}{1}+\del_0\del_1+A^2-A^{-2}+2&\qquad[p]=-1\\
       A^2\twos{p+1}{2}+A^{-2}\twos{p-1}{2}+2\twos{1}{0}+(\del_0+\del_1)\twos{p/2}{1}&\qquad[p]=0\\
       A^2\twos{p+1}{2}+A^{-2}\twos{p-1}{2}+A(\del_0+\del_1)\twos{\frac{p+1}{2}}{1}+\del_0\del_1-A^2+A^{-2}+2 &\qquad[p]=1\\
       A^2\twos{p+1}{2}+A^{-2}\twos{p-1}{2}+2\twos{1}{0}&\qquad[p]=2
    \end{cases}
    \]
    In each case, the argument has the same structure as the proof of Lemmas \ref{lem:easy-prod-sum}.

    For each $[p]\in \{-1,0,1,2\}$ consider the action given in Lemma \ref{lem:sl2z-action} of the matrix $M_{[p]}\in \Lambda$, where
    \[
    M_{-1}=\begin{bmatrix}
        \frac{5-p}{2}&\frac{p-3}{4}\\
        3-p& \frac{p-1}{2}
    \end{bmatrix},\quad
    M_0=\begin{bmatrix}
        1-p/2&p/4\\
        -p&p/2+1
    \end{bmatrix},\quad
    M_1=\begin{bmatrix}
        \frac{3-p}{2}&\frac{p-1}{4}\\
        1-p&\frac{p+1}{2}
    \end{bmatrix},\quad
    M_2=\begin{bmatrix}
        2-p/2&\frac{p-2}{4}\\
        2-p&p/2
    \end{bmatrix}.
    \]
    Finally, to calculate the discrepancies, notice
    \begin{align*}
        \two{p+1}{2}_T&=\begin{cases}
            \twos{p+1}{2} & [p]\in \{-1,0,2\}\\
            \twos{p+1}{2}-2& [p]=1
        \end{cases}\qquad\text{and}\qquad
        \two{p-1}{2}_T=\begin{cases}
            \twos{p-1}{2} & [p]\in \{0,1,2\}\\
            \twos{p-1}{2}-2& [p]=-1
        \end{cases}.\qedhere
    \end{align*}
\end{proof}

\begin{proposition}\label{prop:discrep-p011}
    For $p\geq 0$,
\begin{align*}
    D\fourmat{p+1}{0}{1}{1}&=\del(\sum_{k=1}^p a(p,k)\two{k}{1})
    +\del^2(\sum_{k=0}^pb(p,k)\cdot \overline{T}_k(\two{1}{0}))
    +\del'(\sum_{k=0}^pc(b,k)\cdot \overline{T}_k(\two{1}{0}))
\end{align*}
where $\del=\del_0+\del_1$ and $\del'=\del_0+\del_1+(A+A^{-1})^2$ and
\begin{align*}
    a(p,k)&=\begin{cases}
    A^k[k]_A&2k\leq p\quad\text{and}\quad k\equiv p \mod 2\\
    A^{p-k+1}[p-k+1]_A & 2k> p\quad\text{and}\quad k\equiv p \mod 2\\
    0&\text{otherwise}
\end{cases}\\
b(p,k)&=\begin{cases}
    A^{-k}\left( \delta_{[p-k]}^{-1}\frac{p-k+1}{4}
+\displaystyle{\sum_{h=1}^{\frac{p-k+[p-k]-2}{4}}}([p-k-4h+2]_A+[p-k-4h]_A)
\right)&k\leq p-3\quad\text{and}\quad k\equiv p-3\mod 2\\
0&\text{otherwise}
\end{cases}
\\
c(p,k)&=\begin{cases}
    A^{-k}\left[\frac{p-k+1}{2}\right]_{A^2}&k\leq p-1\quad\text{and}\quad k\equiv p-1\mod 2\\
    0&\text{otherwise}
\end{cases}
\end{align*}

\end{proposition}
\begin{proof}
    By Corollary \ref{cor:sikiwang-3.3}, applying Lemma \ref{lem:easy-discrep} and Lemma \ref{lem:discrep-1p02},
    \begin{align*}
    D\fourmat{p+1}{0}{1}{1}&=A^{-1}\two{1}{0}*D\fourmat{p}{0}{1}{1}-A^{-2}D\fourmat{p-1}{0}{1}{1}+A^{-p-1}D\fourmat{1}{p}{0}{0}-A^{-1}D\fourmat{1}{p}{0}{1}*\two{0}{1}+A^{p-1}D\fourmat{1}{p}{0}{2}\\
    &=
A^{-1}\two{1}{0}*D\fourmat{p}{0}{1}{1}-A^{-2}D\fourmat{p-1}{0}{1}{1}
    -A^{-1}(\delta_{[p]}^0+\delta_{[p]}^2)(\del_0+\del_1)\two{0}{1}\\
    &\qquad+A^{p-1+[p]}(1-\delta_{[p]}^2)(\del_0+\del_1)\two{\frac{p+[p]}{2}}{1}
    +A^{p-1}(\delta_{[p]}^{-1}+\delta_{[p]}^1)(\del_0\del_1+(A+A^{-1})^2)
    \end{align*}
    Using this recurrence relation and \eqref{eq:chev-identity}, this claim can be proved by induction. The base cases $D\lilmat{1}{0}{1}{1}=0$ and $D\lilmat{2}{0}{1}{1}=(\del_0+\del_1)A\twos{1}{1}+\del_0\del_1+(A+A^{-1})^2$ can be verified by hand.
\end{proof}

\bibliographystyle{plain}
\bibliography{ResearchChloe}

\end{document}